\documentclass[11pt,a4paper,leqno,twoside]{amsart}
\usepackage{latexsym,amssymb,amsmath}
\input xy
\xyoption{all}

\def\CC{\mathbb C}

\def\RR{\mathbb R}
\def\HH{\mathbb H}
\def\AA{{\mathbb A}}

\def\OO{\mathbb O}

\def\11{\mathbf 1}
\def\PP{\mathbb P}

\def\e1{\varepsilon_1}
\def\e2{\varepsilon_2}
\def\e3{\varepsilon_3}

\def\P2{{\PP}^2}

\def\00{\underline{0}}
\def\J0{{\cal J}_3(\underline{0})}

\def\PJ0{\PP({\cal J}_3(\underline{0}))}

\def\e{\varepsilon}

\def\AP2{{\AA\PP}^2}
\def\RP2{{\RR\PP}^2}
\def\CP2{{\CC\PP}^2}
\def\HP2{{\HH\PP}^2}
\def\OP2{{\OO\PP}^2}

\newtheorem{theo}{Theorem}[section]
\newtheorem{coro}[theo]{Corollary}
\newtheorem{lemm}[theo]{Lemma}
\newtheorem{prop}[theo]{Proposition}
\theoremstyle{remark}
\newtheorem{rema}[theo]{Remark}
\theoremstyle{definition}

\begin{document}

\title{Demushkin groups and inverse Galois theory for
pro-$p$-groups of finite rank and maximal $p$-extensions}

\author
{I.D. Chipchakov}
\address{Institute of Mathematics and Informatics\\Bulgarian Academy
of Sciences\\Acad. G. Bonchev Str., bl. 8\\1113, Sofia, Bulgaria;
email: chipchak@math.bas.bg} 

\begin{abstract}
This paper proves that if $E$ is a field, such that the Galois group
$\mathcal{G}(E(p)/E)$ of the maximal $p$-extension $E(p)/E$ is a
Demushkin group of finite rank $r(p)_{E} \ge 3$, for some prime
number $p$, then $\mathcal{G}(E(p)/E)$ does not possess nontrivial
proper decomposition groups. When $r(p)_{E} = 2$, it describes the
decomposition groups of $\mathcal{G}(E(p)/E)$. The paper shows that
if $(K, v)$ is a $p$-Henselian valued field with $r(p)_{K} \in
\mathbb N$ and a residue field of characteristic $p$, then $P \cong
\widetilde P$ or $P$ is presentable as a semidirect product $\mathbb
Z _{p} ^{\tau } \rtimes \widetilde P$, for some $\tau \in \mathbb
N$, where $\widetilde P$ is a Demushkin group of rank $\ge 3$ or a
free pro-$p$-group. It also proves that when $\widetilde P$ is of
the former type, it is continuously isomorphic to $\mathcal{G}(K
^{\prime }(p)/K ^{\prime })$, for some local field $K ^{\prime }$
containing a primitive $p$-th root of unity.
\end{abstract}

\maketitle

\section{\bf Introduction and statements of the main results}

Let $P$ be a pro-$p$-group, for some prime number $p$, and let
$r(P)$ be the rank of $P$, i.e. the cardinality of any minimal
system of generators of $P$ as a profinite group. This paper is
devoted to the study of $P$ in case $r(P) \in \mathbb N$ and $P$ is
admissible, i.e. (continuously) isomorphic to the Galois group
$\mathcal{G}(E(p)/E)$ of the maximal $p$-extension $E(p)$ of a field
$E$ (in a separable closure $E _{\rm sep}$ of $E$). Denote for
brevity by $r(p)_{E}$ the rank of $\mathcal{G}(E(p)/E)$. It is known
that if $p = {\rm char}(E)$, then $\mathcal{G}(E(p)/E)$ is a free
pro-$p$-group (cf. \cite{S2}, Ch. II, Proposition~2), so we consider
fields of characteristic different from $p$. Recall that if $r(P) =
1$, then $P$ is admissible if and only if it is isomorphic to the
additive group $\mathbb Z _{p}$ of $p$-adic integers or $p = 2$ and
$P$ is of order $2$ (cf. \cite{Wh}, Theorem~2). Admissible $P$ have
been described, up-to isomorphisms, in the special case where $r(P)
= 2$ and the ground field contains a primitive $p$-th root of unity
(see \cite{Lab2}, page 107, and \cite{EV}). The description relies
on the fact that then $P$ is a free pro-$p$-group or a Demushkin
group unless $p = 2$ and $E$ is a formally real field, in the sense
of Artin-Schreier (cf. \cite{W2}, Lemma~7, \cite{L}, Ch. XI, Sect.
2, and \cite{S1}, Ch. I, 3.3 and 4.5).
\par
\medskip
In this paper, we focus our attention on the case where $r(P) \ge 3$
and the considered ground fields are endowed with $p$-Henselian
valuation. Our starting point is the following result of local
class field theory concerning $r(p)_{F} = r$, for an arbitrary
finite extension $F$ of the field $\mathbb Q _{p}$ of $p$-adic
numbers of degree $N$ (see \cite{Ch} and \cite{S2}, Ch. II,
Theorems~3 and 4):
\par
\medskip
(1.1) (i) $r = N + 2$ in case $F$ contains a primitive $p$-th root
of unity; $r = N + 1$, otherwise.
\par
(ii) The index of $F ^{\ast p}$ in the multiplicative group $F
^{\ast }$ of $F$ is equal to $p ^{r}$.
\par
\medskip\noindent
The structure of $\mathcal{G}(F(p)/F)$ is determined as follows:
\par
\medskip
(1.2) $\mathcal{G}(F(p)/F)$ is a free pro-$p$-group, if $F$ does not
contain a primitive $p$-th root of unity (Shafarevich \cite{Sh}, see
also \cite{S2}, Ch. II, Theorem~3); $\mathcal{G}(F(p)/F)$ is a
Demushkin group, otherwise (Demushkin \cite{D1,D2}, Labute
\cite{Lab2} and Serre \cite{S1}, see also \cite{S2}, Ch. II,
Theorem~4).
\par
\medskip
For convenience of the reader, we recall that an infinite
pro-$p$-group $P$ is said to be a Demushkin group, if the continuous
cohomology group homomorphism $\varphi _{h}\colon \ H ^{1} (P, \mathbb F
_{p}) \to H ^{2}(P, \mathbb F _{p})$ mapping each $g \in H ^{1} (P,
\mathbb F _{p})$ into the cup-product $h \cup g$ is surjective, for
every $h \in H ^{1} (P, \mathbb F _{p}) \setminus \{0\}$, and the
group $H ^{2}(P, \mathbb F _{p})$ is of order $p$ (throughout this
paper, $\mathbb F _{p}$ denotes a field with $p$ elements). The
papers referred to in (1.2) contain a classification, up-to a
continuous isomorphism, of Demushkin pro-$p$-groups of finite ranks,
for each prime $p$. In particular, this classification yields the
following:

\par
\medskip
(1.3) (i) The ranks of finitely-generated Demushkin pro-$p$-groups
are even numbers, provided that $p > 2$;
\par
(ii) For each pair $(d, \theta )$ of positive integers with $2 \mid
d$, there exists a Demushkin group $P _{d,\theta }$, such that $r(P
_{d,\theta }) = d$ and the reduced component of the (continuous)
character group $C(P _{d,\theta })$ is cyclic of order $p ^{\theta
-1}$; when $p
> 2$ or $\theta \neq 2$, $P _{d,\theta }$ is uniquely determined,
up-to an isomorphism;
\par
(iii) For any integer $d \ge 2$, there are pairwise nonisomorphic
Demushkin pro-$2$-groups $D _{n}$, $n \in \mathbb N$, such that $r(D
_{n}) = d$ and the reduced components of $C(D _{n})$ are of order
$2$, for each index $n$.
\par
\medskip
As shown by Pop \cite{Pop}, (2.7), statement (1.2) retains validity,
if $(F, v)$ is a Henselian real-valued field with a residue field
$\widehat F$ of characteristic $p$, a $p$-indivisible value group
$v(F)$ and a finite quotient group $F ^{\ast }/F ^{\ast p}$. These
results relate the admissibility problem for Demushkin groups of
finite ranks with the open question of whether every admissible
Demushkin group $\widetilde P$ of rank $r(\widetilde P) \ge 3$ is
standardly admissible, i.e. $\widetilde P \cong
\mathcal{G}(F(p)/F)$, for some finite extension $F/\mathbb Q _{p}$
(see \cite{E2}, Proposition~8.2). In view of the irreducibility of
the $p ^{n}$-th cyclotomic polynomial over $\mathbb Q _{p}$, for
every $n \in \mathbb N$ (cf. \cite{I}, Ch. 8, Theorem~1), the papers
quoted in (1.2) (ii) give the following necessary conditions for
standard admissibility of $\widetilde P$; in view of (1.1), (1.2)
and (1.3), these conditions are sufficient in case $p > 2$ (see also
\cite{E3}, Theorem~7.3):
\par
\medskip
(1.4) With notations being as in (1.3) (ii), $\widetilde P \cong P
_{d,\theta }$, where $\theta \ge 2$, $d \ge 3$ and $d - 2$ is
divisible by $(p - 1)p ^{(\theta - 2)}$.
\par
\medskip
These results have been extended by Efrat \cite{E1,E2} to the case
of $p$-Henselian fields containing a primitive $p$-th root of
unity (see \cite{MT}, for an earlier result of this kind
concerning the absolute Galois groups of arbitrary Henselian
fields). The purpose of this paper is to extend the scope of (1.2)
along the lines drawn by Efrat (see \cite{E1}, page 216). In order
to simplify the description of the obtained results, consider
first an arbitrary nontrivially valued field $(K, v)$. In what
follows, $O _{v}(K)$ and $\widehat K$ will be the valuation ring
and the residue field of $(K, v)$, respectively, and $K _{h(v)}$
will be a fixed Henselization of $K$ in $K _{\rm sep}$ relative to
$v$ (it is known that $K _{h(v)}$ is uniquely determined by $(K,
v)$, up-to a $K$-isomorphism, see \cite{E3}). Given a prime number
$p \neq {\rm char}(K)$, we denote by $\varepsilon $ some primitive
$p$-th root of unity in $K _{\rm sep}$, and by $G(K)$ the minimal
isolated subgroup of $v(K)$ containing $v(p)$. By definition,
$\Delta (v)$ is the following subgroup of $v(K)$: $\Delta (v) =
G(K)$ in case $\varepsilon \notin K _{h(v)}$; if $\varepsilon \in
K _{h(v)}$, $[K(\varepsilon )\colon K] = m$ and $v ^{\prime }$ is
a valuation of $K(\varepsilon )$ extending $v$, then $\Delta (v)$
is the minimal isolated subgroup of $v(K)$ containing the values
$v ^{\prime }((\varepsilon - c) ^{m})\colon \ c \in O _{v}(K)$, $c
\neq \varepsilon $. It is easy to see that $G(K) = \{0\}$ unless
char$(\widehat K) = p$, and that $\Delta (v)$ does not depend on
the choice of $v ^{\prime }$. With these notation, the main
results of this paper can be stated as follows:

\medskip
\begin{theo}
Let $(K, v)$ be a $p$-Henselian field with {\rm char}$(K) = 0$, {\rm
char}$(\widehat K) = p$ and $r(p)_{K} \in \mathbb N$. Then $\widehat
K$ is perfect and the following conditions hold:
\par
{\rm (i)} $\mathcal{G}(K(p)/K)$ is a free pro-$p$-group if and only
if $\varepsilon \notin K _{h(v)}$ or $\Delta (v) = p\Delta (v)$;
\par
{\rm (ii)} $\mathcal{G}(K(p)/K)$ is isomorphic to a topological
semidirect product $\mathbb Z _{p} ^{\kappa } \rtimes \Psi $, for
some $\kappa \in \mathbb N$ and a free pro-$p$-group $\Psi $, if and
only if $G(K) = pG(K)$, $\varepsilon \in K _{h(v)}$ and the group
$\Delta (v)/p\Delta (v)$ is of order $p ^{\kappa }$.
\end{theo}

\medskip
\begin{theo}
Let $(K, v)$ be a $p$-Henselian field, such that {\rm char}$(K) =
0$, {\rm char}$(\widehat K) = p$ and $G(K) \neq pG(K)$. Then
$r(p)_{K} \in \mathbb N$ if and only if $\widehat K$ is finite,
$G(K)$ is cyclic and, in case $\varepsilon \in K _{h(v)}$, the group
$\Delta (v)/p\Delta (v)$ is finite. When this holds, the following
conditions are equivalent:
\par
{\rm (i)} $\mathcal{G}(K(p)/K) \cong \mathbb Z _{p} ^{\kappa -1}
\rtimes \Psi $, where $\kappa \in \mathbb N$ and $\Psi $ is a
standardly admissible Demushkin group with $r(\Psi ) \ge 3$;
\par
{\rm (ii)} $\varepsilon \in K _{h(v)}$ and $\Delta (v)/p\Delta (v)$
is of order $p ^{\kappa }$.
\par\noindent
In particular, $\mathcal{G}(K(p)/K)$ is a Demushkin group if and
only if $\varepsilon \in K _{h(v)}$ and the group $\Delta (v)/G(K)$
is $p$-divisible.
\end{theo}

\medskip
It is known that if $(E, w)$ is a valued field and $\Omega /E$ is a
Galois extension, then the Galois group $\mathcal{G}(\Omega /E)$
acts transitively on the set $w(\Omega /E)$ of valuations of $\Omega
$ extending $w$, whose value groups are included in a fixed
divisible hull of $w(E)$. Therefore, the decomposition groups of
$\mathcal{G}(\Omega /E)$ attached to $w$, i.e. the stabilizers
Stab$(w ^{\prime }) = \{\psi \in \mathcal{G}(\Omega /E)\colon \ w
^{\prime } \circ \psi = w ^{\prime }\}$, where $w ^{\prime }$ runs
across $w(\Omega /E)$, form a conjugacy class in $\mathcal{G}(\Omega
/E)$. Recall that each admissible pro-$p$-group $P$ is isomorphic to
$\mathcal{G}(L(p)/L)$, for some Henselian field $(L, \omega )$ (and
thereby is realizable as a decomposition group). Indeed, it is known
(see, e.g., \cite{Er}, Ch. 4, Sects. 6 and 7) that for each field
$F$ and ordered abelian group $\Gamma $ there exists a Henselian
field $(K, v)$ with $\widehat K \cong F$ and $v(K) = \Gamma $, so
the noted property of $P$ can be obtained from (2.3) (ii) and (2.7).
Note also that, by the Endler-Engler-Schmidt theorem, see \cite{EE},
for any nontrivial valuation $\omega ^{\prime }$ of $L$, which is
independent of $\omega $, the decomposition group of
$\mathcal{G}(L(p)/L)$ attached to $\omega $ is trivial. Our third
main result proves that if $P$ is an admissible Demushkin group and
$3 \le r(P) < \infty $, then its nontrivial proper subgroups are not
realizable as decomposition groups.

\medskip
\begin{theo}
Let $E$ be a field with a nontrivial Krull valuation $w$. Assume
that $v$ is not $p$-Henselian, $\mathcal{G}(E(p)/E)$ is a Demushkin
group, $r(p)_{E} \in \mathbb N$ and $r(p)_{E} \ge 3$. Then $E(p)$ is
included in any Henselization $E _{h(v)} \subseteq E _{\rm sep}$.
\end{theo}

\medskip
The main results of this paper enable one to show (see Remarks 7.3
and 5.3) that admissible Demushkin pro-$p$-groups of rank $\ge 3$
will be standardly admissible, if the considered ground fields
contain primitive $p$-th roots of unity, and the following open
problem has an affirmative solution:
\par
\medskip
(1.5) Let $F$ be a field, such that $r(p)_{F} \ge 1$ and $F(p)
\subseteq F _{h(\omega )}$, for each nontrivial Krull valuation
$\omega $ of $F$, and suppose that the transcendency degree of $F$
over its prime subfield is finite. Is $\mathcal{G}(F(p)/F)$ a free
pro-$p$-group?
\par
\medskip
It is known (see \cite{Ta}, page 265, and \cite{S1}, Ch. I, 4.2)
that if $F$ is a field with a primitive $p$-th root of unity, then
$\mathcal{G}(F(p)/F)$ is a free pro-$p$-group if and only if
$r(p)_{F} \ge 1$ and the $p$-component Br$(F) _{p}$ of the Brauer
group Br$(F)$ is trivial. When $F(p) = F _{\rm sep}$ and $p \neq
{\rm char}(F)$, the fulfillment of (1.5) implies that Br$(F) =
\{0\}$ in all presently known cases, since then $F$ turns out to be
pseudo algebraically closed, i.e. each geometrically irreducible
affine variety defined over $F$ has an $F$-rational point (see
\cite{FJ}, Theorem~10.17 and Problem~11.5.9, and \cite{S2}, Ch. II,
3.1).

\medskip
The basic notation, terminology and conventions kept in this paper
are standard and virtually the same as in \cite{Ch3}. Throughout,
$\mathbb P$ denotes the set of prime numbers, Brauer and value
groups are written additively, Galois groups are viewed as profinite
with respect to the Krull topology, and by a profinite group
homomorphism, we mean a continuous one. For any field $E$, $E ^{\ast
}$ denotes its multiplicative group, $E ^{\ast n} = \{a ^{n}\colon \
a \in E ^{\ast }\}$, for each $n \in \mathbb N$, $\mathcal{G}_{E} =
\mathcal{G}(E _{\rm sep}/E)$ stands for the absolute Galois group of
$E$, and for each $p \in \mathbb P$, $_{p} {\rm Br}(E) = \{b _{p} \in
{\rm Br}(E)\colon \ pb _{p} = 0\}$ and $P(E) = \{p \in \mathbb
P\colon \ E(p) \neq E\}$. We denote by $s(E)$ the class of
finite-dimensional simple $E$-algebras, $d(E)$ is the subclass of
division algebras from $s(E)$, and for each $A \in s(E)$, $[A]$ is
the similarity class of $A$ in Br$(E)$. As usual, Br$(E ^{\prime
}/E)$ denotes the relative Brauer group of an arbitrary field
extension $E ^{\prime }/E$. We write $\rho _{E'/E}$ for the scalar
extension map of Br$(E)$ into Br$(E ^{\prime })$, and I$(E ^{\prime
}/E)$ for the set of intermediate fields of $E ^{\prime }/E$.

\medskip
The paper is organized as follows: Section 2 includes preliminaries
on Krull valuations used in the sequel. Section 3 contains
characterizations of free pro-$p$-groups and of Demushkin groups of
finite ranks in the class of Galois groups of maximal
$p$-extensions. The proof of Theorems 1.1 and 1.2 is divided into
three parts presented in Sections 4, 5 and 6. This allows us to
determine the structure of $\mathcal{G}(K(p)/K)$ by a reduction to
the special case where $K$ contains a primitive $p$-th root of unity
(see Remark 6.3). Specifically, it becomes clear (Corollary 6.4)
that a pro-$p$-group $P$ of rank $2$ is isomorphic to
$\mathcal{G}(K(p)/K)$, for some $p$-Henselian field $(K, v)$, if and
only if $P$ is a free pro-$p$-group or a Demushkin group. Theorem
1.3 is deduced from Theorems 1.1 and 1.2 in Section 7, where we also
present a description of the decomposition groups of Demushkin
groups of rank $2$.

\medskip

\section{\bf Preliminaries on Henselian $\Omega $-valuations}

Let $K$ be a field with a nontrivial valuation $v$, $O _{v}(K) = \{a
\in K\colon \ v(a) \ge 0\}$ the valuation ring of $(K, v)$, $M
_{v}(K) = \{\mu \in K\colon \ v(\mu ) > 0\}$ the maximal ideal of $O
_{v}(K)$, $v(K)$ and $\widehat K = O _{v}(K)/M _{v}(K)$ the value
group and the residue field of $(K, v)$, respectively. For each
$\gamma \in v(K)$, $\gamma \ge 0$, we denote by $\nabla _{\gamma
}(K)$ the set $\{\lambda \in K\colon \ v(\lambda - 1) > \gamma \}$.
As usual, the completion of $K$ relative to (the topology induced
by) $v$ is denoted by $K _{v}$, and is considered with its valuation
$\bar v$ continuously extending $v$. Whenever $v(E)$ is Archimedean,
i.e. it is embeddable as an ordered subgroup in the additive group
$\mathbb R$ of real numbers, we identify $v(E)$ with its isomorphic
copy in $\mathbb R$. In what follows, Is$_{v} (K)$ denotes the set
of isolated subgroups of $v(K)$ different from $v(K)$. It is
well-known (cf. \cite{B1}, Ch. VI, Sect. 4.3) that each $H \in {\rm
Is}_{v}(K)$ is a pure subgroup of $v(K)$, the ordering of $v(K)$
induces canonically on $v(K)/H$ a structure of an ordered group, and
one can naturally associate with $v$ and $H$ a valuation $v _{H}$ of
$K$ with $v _{H} (K) = v(K)/H$. Unless specified otherwise, $K _{H}$
will denote the residue field of $(K, v _{H})$, $\eta _{H}$ the
natural projection $O _{v _{H}}(K) \to K _{H}$, and $\hat v _{H}$
the valuation of $K _{H}$ induced canonically by $v$ and $H$. The
valuations $v$, $v _{H}$ and $\hat v _{H}$ are related as follows:
\par
\medskip
(2.1) (i) $\hat v _{H}(K _{H}) = H$, $\widehat K _{H}$ is isomorphic
to $\widehat K$ and $\eta _{H}$ induces a surjective homomorphism of
$O _{v}(K)$ upon $O _{\hat v _{H}}(K _{H})$ (cf. \cite{E3},
Proposition~5.2.1).
\par
(ii) For each $H \in {\rm Is}_{v}(K)$, $v _{H}$ induces on $K$ the
same topology as $v$; the mapping of Is$_{v}(K)$ on the set $V _{v}$
of proper subrings of $K$ including $O _{v}$, by the rule $X \to v
_{X}(K)$, $X \in {\rm Is}_{v}(K)$, is an inclusion-preserving
bijection.
\par
(iii) If $v(K)$ properly includes the union $H(K)$ of the groups $H
\in {\rm Is}_{v}(K)$, then $v(K)/H(K)$ is Archimedean (cf.
\cite{E3}, Theorem~2.5.2).
\par
\medskip
We say that the valuation $v$ is $\Omega $-Henselian, for a given
normal extension $\Omega /K$, if $v$ is uniquely, up-to an
equivalence, extendable to a valuation $v _{L}$ on each $L \in
I(\Omega /K)$. In order that $v$ is $\Omega $-Henselian, it is
necessary and sufficient that the Hensel-Rychlik condition holds
(cf. \cite{E3}, Sect. 18.1):
\par
\medskip
(2.2) Given a polynomial $f(X) \in O _{v}(K) [X]$, which fully
decomposes in $\Omega [X]$, and an element $a \in O _{v}(K)$, such
that $2v(f ^{\prime }(a)) < v(f(a))$, where $f ^{\prime }$ is the
formal derivative of $f$, there is a zero $c \in O _{v}(K)$ of $f$
satisfying the equality $v(c - a) = v(f(a)/f ^{\prime }(a))$.
\par
\medskip
If $v$ is $\Omega $-Henselian and $L \in I(\Omega /K)$, then $v
_{L}$ is $\Omega $-Henselian. In this case, we denote by $\widehat
L$ the residue field of $(L, v _{L})$, put $v(L) = v _{L}(L)$, and
write $L _{v}$ for the completion of $L$ with respect to $v _{L}$.
When $v(K)$ is non-Archimedean, the $\Omega $-Henselian property can
be also characterized as follows:
\par
\medskip
\begin{prop}
Let $(K, v)$ be a valued field, and let $H \in {\rm Is}_{v}(K)$.
Then $v$ is $\Omega $-Henselian if and only if $v _{H}$ is $\Omega
$-Henselian and $\hat v _{H}$ is $\widehat \Omega $-Henselian.
\end{prop}

\begin{proof}
Suppose first that $v _{H}$ is $\Omega $-Henselian and $\hat v _{H}$
is $\widehat \Omega $-Henselian. Fix a monic polynomial $f(X) \in O
_{v}(K) [X]$, which fully decomposes over $\Omega $ and has a simple
zero $\beta \in O _{v}(K)$ modulo $M _{v}(K)$. Denote by $\hat f
_{H}$ the reduction of $f$ modulo $M _{v _{H}}(K)$. Then $\beta \in
O _{v _{H}}(K)$, $\hat f _{H}(X) \in O _{\hat v _{H}}(K _{H}) [X]$,
and the residue class $\hat \beta _{H} \in K _{H}$ is a simple zero
of $\hat f _{H}$ modulo $M _{\hat v _{H}}(K _{H})$. Hence, by the
$\widehat \Omega $-Henselian property of $\hat v _{H}$, $\hat f
_{H}$ has a simple zero $\tilde \alpha _{H} \in O _{\hat v _{H}}(K
_{H})$, such that $\hat \beta _{H} - \tilde \alpha _{H} \in M _{\hat
v _{H}}(K _{H})$. Now take a preimage $\alpha _{H}$ of $\tilde
\alpha _{H}$ in $O _{v _{H}}(K)$. The preceding observations
indicate that $f(\alpha _{H}) \in M _{v _{H}}(K)$ and $f ^{\prime
}(\alpha _{H}) \notin M _{v _{H}}(K)$. Since $v _{H}$ is $\Omega
$-Henselian, it follows from (2.2) that $O _{v _{H}}(K)$ contains a
simple zero $\alpha $ of $f$ with the property that $\alpha - \alpha
_{H} \in M _{v _{H}}(K)$. As $O _{v}(K)$ is an integrally closed
ring, it is now easy to see that $\alpha \in O _{v}(K)$ and $\alpha
- \beta \in M _{v}(K)$, which proves that $v$ is $\Omega
$-Henselian.
\par
Conversely, assume that $v$ is $\Omega $-Henselian, fix a finite
field extension $L/K$, denote by $O _{H}(L)$ the integral closure of
$O _{v _{H}}(K)$ in $L$, and let $w _{1}$ and $w _{2}$ be valuations
of $L$ extending $v _{H}$. The $\Omega $-Henselian property of $v$
ensures that $O _{v}(L)$ equals the integral closure of $O _{v}(K)$
in $L$. At the same time, one concludes that $O _{v}(L) \subseteq O
_{H}(L) \subseteq O _{w _{j}}(L)$ and $O _{w _{j}}(L) = O _{v _{L},H
_{j}}(L)$, for some $H _{j} \in {\rm Is}_{v _{L}}(L)$, $j = 1, 2$.
Since the set of valuation subrings of $L$ including $O _{v}(L)$ is
a chain with respect to set-theoretic inclusion, one can assume
without loss of generality that $O _{w _{1}}(L) \subseteq O _{w
_{2}}(L)$, whence $H _{1} \subseteq H _{2}$ (see \cite{B1},Ch. VI,
Sect. 4, Proposition~4). As $H _{j} \in {\rm Is}_{v _{L}}(L)$, $j =
1, 2$, this means that $H _{2}/H _{1}$ is a torsion-free group,
whereas Ostrowski's theorem implies that $H$ is a subgroup of $H
_{j}$ of index dividing $[L\colon K]$, for each index $j$.
Therefore, $H _{2}/H _{1}$ is a homomorphic image of $H _{2}/H$,
whence its order also divides $[L\colon K]$. The obtained results
prove that $H _{1} = H _{2}$ and $v _{H}$ is $\Omega $-Henselian. It
remains to be seen that $\hat v _{H}$ is $\widehat \Omega
$-Henselian. Let $\tilde h(X) = X ^{n} + \sum _{i=0} ^{n-1} \tilde a
_{i}X ^{n-i} \in O _{\hat v _{H}}(K _{H}) [X]$ have a simple zero
$\tilde \eta \in O _{\hat v _{H}}(K _{H})$ modulo $M _{\hat v _{H}}(K
_{H})$, and let $a _{i}$ be a preimage of $\tilde a _{i}$ in $O
_{\hat v _{H}}(K _{H})$, for $i = 1, \dots , n - 1$. Then $h(X) = X
^{n} + \sum _{i=0} ^{n-1}a _{i}X ^{n-i}$ lies in $O _{v}(K) [X]$ and
$O _{v}(K)$ contains any preimage $\eta $ of $\tilde \eta $ in $O
_{v _{H}}(K)$. Moreover, by the $\Omega $-Henselity of $v$, the
coset $\eta + M _{v}(K)$ contains a simple zero $\xi $ of $h$. This
implies $\tilde h(\hat \xi _{H}) = 0$, $\tilde h ^{\prime }(\hat \xi
_{H}) \neq 0$ and $\hat \xi _{H} - \tilde \eta \in M _{\hat v
_{H}}(K _{H})[X]$, which proves that $\hat v _{H}$ is $\widehat
\Omega $-Henselian.
\end{proof}

\medskip
A finite extension $R$ of $K$ in $\Omega $ is said to be inertial,
if $R$ has a unique, up-to an equivalence, valuation $v _{R}$
extending $v$, the residue field $\widehat R$ of $(R, v _{R})$ is
separable over $\widehat K$, and $[R\colon K] = [\widehat R\colon
\widehat K]$. When $v$ is $\Omega $-Henselian, these extensions have
the following frequently used properties (see \cite{JW}, page 135
and Theorems~2.8 and 2.9, for the case where $v$ is Henselian):
\par
\medskip
(2.3) (i) An inertial extension $R$ of $K$ in $\Omega $ is Galois if
and only if $\widehat R/\widehat K$ is Galois. When this holds,
$\mathcal{G}(R/K)$ and $\mathcal{G}(\widehat R/\widehat K)$ are
canonically isomorphic.
\par
(ii) The set of inertial extensions of $K$ in $\Omega $ is closed
under the formation of subextensions and finite compositums. The
compositum $\Omega _{0}$ of these extensions is Galois over $K$ with
$\mathcal{G}(\Omega _{0}/K) \cong \mathcal{G}(\widehat \Omega
/\widehat K)$.
\par
\medskip

\begin{prop}
Let $(K, v)$ be a nontrivially valued field and $(K _{h(v)}, \sigma
)$ a Henselization of $(K, v)$. Then the residue field of $(K
_{h(v)}, \sigma _{H})$ is isomorphic to $K _{H}$, for each $H \in
{\rm Is}_{v}(K)$. Moreover, if $v(p) \in H$, for some $p \in \mathbb
P$, then $K _{h(v)}$ contains a primitive $p$-th root of unity if
and only if so does $K _{H}$.
\end{prop}

\medskip
\begin{proof}
Fix Henselizations $(K _{h(v _{H})}, \omega _{H})\colon \ K _{h(v
_{H})} \subseteq K _{\rm sep}$, and $(\widetilde \Phi , \tilde v
_{H})\colon \ \widetilde {\Phi } \subseteq K _{H,{\rm sep}}$, of
$(K, v _{H})$ and $(K _{H}, \hat v _{H})$, respectively, and denote
by $\Sigma (H)$ the compositum of the inertial extensions of $K
_{h(v _{H})}$ in $K _{\rm sep}$ relative to $\sigma _{H}$. Also, let
$\Phi $ be the preimage of $\widetilde {\Phi }$ under the bijection
$I(\Sigma (H))/K _{h(v _{H})}) \to I(\widehat K _{\rm sep}/\widehat
K)$, canonically induced by the natural homomorphism of $O _{\omega
_{H}}(\Sigma (H))$ into $\widehat K _{\rm sep}$, and let $\varphi
_{H}$ be the valuation of $\Phi $ extending $\omega _{H}$. It
follows from Proposition 2.1 and \cite{E3}, Theorem~15.3.5, that $(K
_{h(v _{H})}, \omega _{H})$ can be chosen so that $K _{h(v _{H})}
\subseteq K _{h(v)}$ and $\omega _{H}$ is induced by $\sigma _{H}$.
It is easily verified that $\widetilde \Phi $ is the residue field
of $(\Phi , \varphi _{H})$ and there exists a valuation $\phi $ of
$\Phi $, such that $\phi _{H} = \varphi _{H}$ and $\hat \phi _{H} =
\tilde v _{H}$. Moreover, it follows from Proposition 2.1, the
definition of $\widetilde \Phi $ and the observation concerning $(K
_{v _{H}}, \omega _{H})$ that $(\Phi , \phi )$ is a Henselization of
$(K, v)$. In view of \cite{E3}, Theorem~15.3.5, this proves the
former assertion of Proposition 2.2. The rest of our proof relies on
the well-known fact that a field $F$ with a Henselian valuation $f$
contains a primitive $p$-th root of unity, for a given $p \in
\mathbb P$, $p \neq {\rm char}(\widehat F)$, if and only if
$\widehat F$ contains such a root. When $v(p) \in H$, this applies
to $(\Phi , \varphi _{H})$, so the latter assertion of Proposition
2.2 can be viewed as a consequence of the former one.
\end{proof}

\medskip
The following lemma plays a major role in the study of
$\mathcal{G}(K(p)/K)$ when $(K, v)$ is $p$-Henselian (i.e.
$K(p)$-Henselian, see also (2.4)). For convenience of the reader, we
prove the lemma here (referring to \cite{LR}, for a much more general
result in the case where $K$ is dense in its Henselizations).

\medskip
\begin{lemm}
Let $(K, v)$ be a valued field and $\widetilde F$ an extension of
$\widehat K$ in $\widehat K(p)$ of degree $p \in \mathbb P$. Then
$\widetilde F$ has an inertial lift in $K(p)$ over $K$, i.e. there
exists an inertial extension $F$ of $K$ in $K(p)$, such that
$\widehat F \cong \widetilde F$ over $\widehat K$.
\end{lemm}

\begin{proof}
If char$(K) = p$, then our assertion is an easy consequence of the
Artin-Schreier theory, so we assume further that char$(K) = p
^{\prime } \neq p$. Suppose first that $v$ is of height $n \in
\mathbb N$. When $n = 1$, our assertion is a special case of
Grunwald-Wang's theorem (cf. \cite{LR}). Proceeding by induction on
$n$, we prove the statement of the lemma, under the hypothesis that
$n \ge 2$ and it holds for valued fields of heights $< n$. Let $H$
be the maximal group from Is$_{v}(K)$. Then the valuation $\hat v
_{H}$ of $K _{H}$ is of height $n - 1$ and, by the inductive
hypothesis, $\widetilde F$ has an inertial lift $\widetilde F _{H}
\subseteq K _{H}(p)$ over $K _{H}$ relative to $\hat v _{H}$. Note
also that if $F \in I(K(p)/K)$ is an inertial lift of $\widetilde F
_{H}$ over $K$ relative to $v _{H}$, then it is an inertial lift of
$\widetilde F$ over $K$ with respect to $v$ as well.
\par
In order to complete the proof of Lemma 2.3 it remains to be seen
that it reduces to the special case in which $v$ is of finite
height. Let $\mathbb F$ be the prime subfield of $K$. Clearly, it
suffices to show that one may consider only the case where
$K/\mathbb F$ has finite transcendency degree. Fix a Henselization
$(K _{h(v)}, \sigma )$ of $(K, v)$ (with $K _{h(v)} \in I(K _{\rm
sep}/K)$) and an inertial lift $H(\widetilde F) \in I(K _{\rm sep}/K
_{h(v)})$ of $\widetilde F$ over $K _{h(v)}$. Let $f(X) = X ^{p} +
\sum _{i=1} ^{p} c _{i}X ^{p-i}$ be the minimal polynomial over $K
_{h(v)}$ of some primitive element $\eta _{1}$ of $H(\widetilde F)/K
_{h(v)}$ lying in $O _{\sigma }(H(\widetilde F))$. Denote by $S$ the
set $\{\eta _{i}\colon \ i = 1, \dots , p\}$ of zeroes of $f(X)$ in
$K _{\rm sep}$, and by $T$ the set of $K _{h(v)}$-coordinates of the
zeroes $\eta _{i}\colon \ i = 2, \dots , p$, with respect to the $K
_{h(f)}$-basis $\eta _{1} ^{j-1}$, $j = 1, \dots , p$, of
$H(\widetilde F)$. Our choice of $\eta _{1}$ ensures the inclusion
$Y _{f} = \{c _{i}\colon \ i = 1, \dots , p\} \subset O _{\sigma }(K
_{h(v)})$. As the valued extension $(K _{h(v)}, \sigma )/(K, v)$ is
immediate, $O _{v}(K)$ contains a system of representatives $Y _{f}
^{\prime } = \{c _{i} ^{\prime }\colon \ i = 1, \dots , p\}$ of the
elements of $Y _{f}$ modulo the ideal $M _{\sigma }(K _{h(v)})$. Let
$S ^{\prime }$ and $T ^{\prime }$ be the sets of coefficients of the
minimal (monic) polynomials over $K$ of the elements of $Y _{f}$ and
$T$, respectively, and let $\Lambda _{0}$ be the extension of
$\mathbb F$ generated by the union $Y _{f} \cup Y _{f} ^{\prime }
\cup S \cup S ^{\prime } \cup T \cup T ^{\prime }$. Assume that
$\Lambda $ is the algebraic closure of $\Lambda _{0}$ in $K _{\rm
sep}$, $\Phi = K \cap \Lambda $, $\Phi ^{\prime } = K _{h(v)} \cap
\Lambda $, and $\psi $ and $\psi ^{\prime }$ are the valuations
induced by $\sigma $ on $\Phi $ and $\Phi ^{\prime }$, respectively.
It is easily verified that the polynomial $g(X) = X ^{p} + \sum
_{i=1} ^{p} c _{i} ^{\prime }X ^{p-i}$ lies in $O _{\psi }(\Phi )
[X]$ and the root field $\widetilde F _{0}$, over $\widehat \Phi $,
of its reduced polynomial $\hat g$ modulo $M _{\psi }(\Phi )$ is a
cyclic extension of $\widehat \Phi $ of degree $p$. Observing that
$\psi $ is of height $\le 1 + d$, where $d$ is the transcendency
degree of $\Phi /\mathbb F$, one concludes that $\widetilde F _{0}$
has an inertial lift $F _{0} \in I(\Phi (p)/\Phi )$ over $\Phi $.
Since $\hat g$ remains irreducible over $\widehat K$, the compositum
$F = F _{0}K$ is an inertial lift of $\widetilde F$ over $K$, so
Lemma 2.3 is proved.
\end{proof}

\par
\medskip
When $(K, v)$ is a $p$-Henselian field, for a given $p \in \mathbb
P$, (2.3) and Lemma 2.3, combined with Galois theory and the
subnormality of proper subgroups of finite $p$-groups (cf. \cite{L},
Ch. I, Sect. 6), imply the following assertions:
\par
\medskip
(2.4) (i) Each finite extension $\widetilde F$ of $\widehat K$ in
$\widehat K(p)$, possesses a uniquely determined, up-to a
$K$-isomorphism, inertial lift over $K$ relative to $v$;
\par
(ii) The residue field of $(K(p), v _{K(p)})$ is $\widehat
K$-isomorphic to $\widehat K(p)$.
\par
\medskip
Statements (2.4) and the following assertion reduce the study of a
number of algebraic properties of maximal $p$-extensions of
$p$-Henselian fields to the special case of Henselian ground fields:
\par
\medskip
(2.5) Let $(K, v)$ be a $p$-Henselian field, $(K _{h(v)}, \sigma )$
a Henselization of $(K, v)$, and $U$ the compositum of inertial
extensions of $K$ in $K(p)$ relative to $v$. Then $UK _{h(v)}$ is $K
_{h(v)}$-isomorphic to the tensor product $U \otimes _{K} K
_{h(v)}$, and equals the compositum of inertial extensions of $K
_{h(v)}$ in $K _{h(v)}(p)$ relative to $\sigma $.
\par
\medskip\noindent
Statement (2.5) is implied by (2.4) and the fact that $(K _{h(v)},
\sigma )$ is immediate over $(K, v)$. In the sequel, we will also
need the following characterization of finite extensions of $K _{v}$
in $K _{v}(p)$, in case $v$ is $p$-Henselian with $v(K)$ Archimedean
(cf. \cite{B1}, Ch. VI, Sect. 8.2, and \cite{JW}, page 135):
\par
\medskip
(2.6) (i) Every finite extension $L$ of $K _{v}$ in $K _{v,{\rm
sep}}$ is $K _{v}$-isomorphic to $\widetilde L \otimes _{K} K _{v}$
and $\widetilde L _{v}$, where $\widetilde L$ is the separable
closure of $K$ in $L$. The extension $L/K _{v}$ is Galois if and
only if so is $\widetilde L/K$; when this holds, $\mathcal{G}(L/K
_{v})$ and $\mathcal{G}(\widetilde L/K)$ are canonically
isomorphic.
\par
(ii) $K _{\rm sep} \otimes _{K} K _{v}$ is a field and there are
canonical isomorphisms $K _{\rm sep} \otimes _{K} K _{v} \cong K
_{v,{\rm sep}}$ and $\mathcal{G}_{K} \cong \mathcal{G}_{K _{v}}$.
\par
\medskip\noindent
For example, when char$(K) = 0$, $\widehat K$ is finite,
char$(\widehat K) = p$ and the minimal group $G(K) \in {\rm
Is}_{v}(K)$ containing $v(p)$ is cyclic, (2.6) allows us to
determine the structure of $\mathcal{G}(K(p)/K)$ in accordance with
(1.1) and (1.2). As noted in the Introduction, the concluding result
of this Section enables one to prove that admissible pro-$p$-groups
are isomorphic to decomposition groups; this result can be deduced
from Galois theory and the main result of \cite{MT}:
\par
\medskip
(2.7) For each Henselian field $(K, v)$, there is $R \in I(K _{\rm
sep}/K)$, such that $\mathcal{G}_{R} \cong \mathcal{G}_{\widehat
K}$, $v(R)$ is divisible and finite extensions of $R$ in $K _{\rm
sep}$ are inertial.

\medskip

\section{\bf Free pro-$p$-groups and Demushkin groups in the class of
Galois groups of maximal $p$-extensions}

The purpose of this Section is to characterize the groups pointed
out in its title. Our argument relies on the following two lemmas.

\medskip
\begin{lemm}
Let $A$ be an abelian torsion $p$-group, $\mu $ a positive integer
dividing $p - 1$, $\varphi $ an automorphism of $A$ of order $\mu $,
and $\varepsilon _{\mu }$ a primitive $\mu $-th root of unity in the
ring $\mathbb Z _{p}$ of $p$-adic integers. Then $\varphi $ and the
mapping $\varphi ^{u} \to \varepsilon _{\mu } ^{u}$, $u = 0, \dots ,
\mu - 1$, induce canonically on $A$ a structure of a $\mathbb Z
_{p}$-module. This module decomposes into a direct sum $A = \oplus
_{u=0} ^{\mu -1} A _{u}$, where $A _{u} = \{a _{u} \in A\colon \
\varphi (a _{u}) = \varepsilon ^{u}a _{u}\}$, for each index $u$.
\end{lemm}

\begin{proof}
The scalar multiplication $\mathbb Z _{p} \times A \to A$ is
uniquely defined by the group operation in $A$ and the rule $z.a =
0\colon \ a \in A$, $p ^{k}a = 0$, $z \in p ^{k}\mathbb Z _{p}$. For
each $n \in \mathbb N$, denote by $s _{n}$ the integer satisfying
the conditions $\varepsilon _{\mu } - s _{n} \in p ^{n}\mathbb Z
_{p}$ and $0 \le s _{n} \le p ^{n} - 1$. It is easily seen that
$\varepsilon _{\mu } ^{u}a = s _{n}a$ whenever $a \in A$, $n \in
\mathbb N$ and $p ^{n}a = 0$. Note also that the element $a _{u} =
\sum _{u'=0} ^{\mu -1} \varepsilon _{\mu } ^{(\mu -u)u'}\varphi
^{u'} (a)$ lies in $A _{u}$, for $u = 0, \dots , \mu - 1$. Since the
matrix $(z _{ij}) = (\varepsilon _{\mu } ^{(1-i)(j-1)})$, $1 \le i,
j \le \mu $, is invertible in the ring $M _{\mu }(\mathbb Z _{p})$,
this implies that the $\mathbb Z _{p}$-submodule of $A$ generated by
the elements $a _{u}$, $u = 0, \dots , \mu - 1$, contains $\varphi
^{u'-1}(a)$, for $u ^{\prime } = 0, \dots , \mu - 1$, which
completes the proof of Lemma 3.1.
\end{proof}

\medskip
\begin{lemm}
Let $E$ be a field with $r(p)_{E} \in \mathbb N$, for some $p \in
P(E)$, $p \neq {\rm char}(E)$, and let $L$ be an extension of $E$ in
$E(p)$ of degree $p$. Assume that $\varepsilon $ is a primitive
$p$-th root of unity in $E _{\rm sep}$, $\varphi $ is a generator of
$\mathcal{G}(E(\varepsilon )/E)$, and $E _{\varepsilon } = \{a \in
E(\varepsilon ) ^{\ast }\colon \ \varphi (a)a ^{-s} \in
E(\varepsilon ) ^{\ast p}\}$, where $s$ is an integer satisfying the
equality $\varepsilon ^{s} = \varphi (\varepsilon )$. Put $N(L/E)
_{\varepsilon } = E _{\varepsilon } \cap N(L(\varepsilon
)/E(\varepsilon ))$ and denote by $p ^{t}$ the index of
$E(\varepsilon ) ^{\ast p}$ in $N(L/E) _{\varepsilon }$. Then:
\par
{\rm (i)} Each $\lambda \in N(L/E) _{\varepsilon }$ is presentable
as a product $\lambda = N _{E(\varepsilon )} ^{L(\varepsilon )}
(\tilde \lambda )\lambda _{0} ^{p}$, for some $\tilde \lambda \in L
_{\varepsilon }$, $\lambda _{0} \in E(\varepsilon ) ^{\ast p}$;
\par
{\rm (ii)} $r(p)_{L} = r(p)_{E} + (p - 1)(t - 1)$;
\par
\smallskip
{\rm (iii)} $E _{\varepsilon } \subseteq N(L(\varepsilon
)/E(\varepsilon ))$ if and only if $r(p)_{L} = 1 + p(r(p)_{E} - 1)$.
\end{lemm}

\begin{proof}
Fix a generator $\psi $ of $\mathcal{G}(L/E)$, denote by $\bar \psi
$ the $E(\varepsilon )$-automorphism of $L(\varepsilon )$ extending
$\psi $, set $[E(\varepsilon )\colon E] = m$, and for each $\lambda
\in L(\varepsilon )$, let $\rho (\lambda ) = \prod _{i=0} ^{m-1}
\bar \varphi ^{i} (\lambda ) ^{l ^{i}}$, where $\bar \varphi $ is
the $L$-automorphism of $L(\varepsilon )$ extending $\varphi $, and
$l$ is an integer with $0 < l \le p - 1$ and $sl \equiv 1 ({\rm mod}
\ p)$. It is easily verified that $\rho (a)a ^{-m} \in E(\varepsilon
) ^{\ast p}$, for each $a \in E _{\varepsilon }$. Note also that
$L(\varepsilon )/E$ is cyclic, whence $N _{E(\varepsilon )}
^{L(\varepsilon )} (\bar \varphi (u)) = \varphi (N _{E(\varphi )}
^{L(\varepsilon )} (u))$, for each $u \in L(\varepsilon )$. This
proves Lemma 3.2 (i).
\par
Lemma 3.2 (iii) is implied by Lemma 3.2 (ii), so we turn to the
proof of Lemma 3.2 (ii). It is easily obtained from Kummer theory and \cite{Wh}, Theorem~2,
that $t = 0$ if and only if $p = 2$, $E$ is a Pythagorean field
(i.e. formally real with $E ^{\ast 2}$ closed under addition) and $L
= E(\sqrt{-1})$. When this holds, it can be deduced from Hilbert's
Theorem 90 (cf. \cite{L}, Ch. VIII, Sect. 6) that $L ^{\ast } = K
^{\ast }L ^{\ast 2}$ and $L ^{\ast 2} \cap K ^{\ast } = K ^{\ast 2}
\cup -1.K ^{\ast 2}$. This implies that $r(2)_{L} = r(2)_{E} - 1$,
as claimed by Lemma 3.2. Henceforth, we assume that $t > 0$. Suppose
first that $r(p)_{E} = 1$. Then it follows from \cite{Wh},
Theorem~2, that $\mathcal{G}(E(p)/E) \cong \mathbb Z _{p}$ unless $p
= 2$ and $E$ is Pythagorean with a unique ordering (i.e. with $E
^{\ast } = E ^{\ast 2} \cup -1.E ^{\ast 2}$). This, combined with
Albert's theorem (cf. \cite{A1}, Ch. IX, Sect. 6) or \cite{Ch3},
Lemma~3.5, proves our assertion. Henceforth, we assume that
$r(p)_{E} \ge 2$. It follows from Kummer theory and Albert's theorem
that $E _{\varepsilon }/E(\varepsilon ) ^{\ast p}$ has dimension
$r(p)_{E}$ as an $\mathbb F _{p}$-vector space. More precisely, it
is easily verified that there exist $a _{1}, \dots , a _{r(p)_{E}}
\in E _{\varepsilon }$, such that the cosets $a _{j}E(\varepsilon )
^{\ast p}$, $j = 1, \dots , r(p)_{E}$, form a basis of $E
_{\varepsilon }/E(\varepsilon ) ^{\ast p}$ and the following
conditions hold:
\par
\medskip
(3.1) (i) $L(\varepsilon )$ is generated over $E(\varepsilon
)$ by a $p$-th root $\xi _{1}$ of $a _{1}$; moreover, $a _{1}$
can be chosen so that $\xi _{1} \in L _{\varepsilon }$;
\par
(ii) There is an index $\bar t\ge 2$, such that $a _{i} \in
N(L(\varepsilon )/E(\varepsilon ))$, $i = 2, \dots , \bar t$, and
$a _{i} \notin N(L(\varepsilon )/E(\varepsilon ))$, $i > t$;
\par
(iii) $a _{t+1} = -a _{1}$, provided that $p = 2$, $-1 \notin E
^{\ast 2}$ and $a _{1} \notin N(L/E)$.
\par
\medskip\noindent
In view of Lemma 3.2 (i), $L _{\varepsilon }$ contains elements $\xi
_{1}, \dots , \xi _{t} \in L _{\varepsilon }$, such that $\xi _{1}
^{p}a _{1} ^{-1} \in E(\varepsilon ) ^{\ast p}$ and $N
_{E(\varepsilon )} ^{L(\varepsilon )} (\xi _{i})a _{i} ^{-1} \in
E(\varepsilon ) ^{\ast p}$, for $i = 2, \dots , \bar t$. For each
index $i \ge 2$, put $\xi _{i,0} = \xi _{i}$, and $\xi _{i,j} = \bar
\psi (\xi _{i,(j-1)})\xi _{i,(j-1)} ^{-1}$, for $j = 2, \dots , (p -
1)$. Using repeatedly Hilbert's Theorem 90 (as in the proof of
implication (cc)$\to $(c) of \cite{Ch3}, (5.3)), one concludes that
$L _{\varepsilon }/L(\varepsilon ) ^{\ast p}$ contains as an
$\mathbb F _{p}$-basis the following set $S$ of cosets:
\par
\medskip
(3.2) (i) If $p > 2$ or $-1 \notin E ^{\ast 2}$, then $S$ consists
of $\xi _{1}L(\varepsilon ) ^{\ast p}$, $\xi _{i,j}L(\varepsilon )
^{\ast p}$, $2 \le i \le \bar t$, $0 \le j \le (p - 1)$, and $a
_{u}L(\varepsilon ) ^{\ast p}$, $u > t$;
\par
(ii) When $p = 2$ and $a _{1} \notin N(L/E)$, $S$ is formed by $\xi
_{1}L(\varepsilon ) ^{\ast 2}$, $\xi _{i,j}L(\varepsilon ) ^{\ast
2}$, $2 \le i \le \bar t - 1$, $0 \le j \le 1$, and $a
_{u}L(\varepsilon ) ^{\ast 2}$, $u \ge \bar t$;
\par
(iii) If $p = 2$, $-1 \notin E ^{\ast 2}$ and there is $\xi _{1}
^{\prime } \in L$ of norm $N_{E}^{L}(\xi _{1} ^{\prime }) = a _{1}$,
then $S$ equals $\xi _{1} ^{\prime }L(\varepsilon ) ^{\ast
2}$, $\xi _{i,j}L(\varepsilon ) ^{\ast 2}$, $2 \le i \le \bar t$, $0
\le j \le 1$, and $a _{u}L(\varepsilon ) ^{\ast 2}$, $u > \bar t$.
\par
\medskip
Note finally that, by \cite{A1}, Ch. IX, Theorem~6, and Kummer
theory, $r(p)_{L}$ equals the dimension of $L _{\varepsilon
}/L(\varepsilon ) ^{\ast p}$ as an $\mathbb F _{p}$-vector space.
This implies that the index $\bar t$ in (3.1) is equal to $t$ in
cases (i) and (iii) of (3.2), and in case (3.2) (ii), $\bar t = t +
1$. Therefore, Lemma 3.2 (ii) can be deduced from (3.2).
\end{proof}

\begin{coro}
Let $E$ be a field with $r(p)_{E} \in \mathbb N$, for some $p \in
P(E)$. Then $\mathcal{G}(E(p)/E)$ is a free pro-$p$-group if and
only if $r(p)_{L} = 1 + [L\colon E]$, for every finite extension $L$
of $E$ in $E(p)$.
\end{coro}

\begin{proof}
This follows from Lemma 3.2 (iii), Galois theory and the
characterization of free pro-$p$-groups of finite rank by the
indices of their open subgroups (cf. \cite{S1}, Ch. I, 4.2).
\end{proof}

\medskip
Assume that $E$ is a field with $r(p)_{E} \le \infty $, for some $p
\in P(E)$, $p \neq {\rm char}(E)$, take $\varepsilon $, $\varphi $,
$s$ and $l$ as in Lemma 3.2, and put $m = [E(\varepsilon )\colon
E]$. In view of Lemma 3.1, then the action of
$\mathcal{G}(E(\varepsilon )/E)$ on $E(\varepsilon )$ canonically
induces on $E(\varepsilon )$, $E(\varepsilon ) ^{\ast }$ and
Br$(E(\varepsilon )) _{p}$ structures of modules over the group ring
$\mathbb Z[\mathcal{G}(E(\varepsilon )/E)]$. Note further that the
$\mathbb Z _{p}$-submodules Br$(E(\varepsilon )) _{p,j}$ $j = 0,
\dots , m - 1$, of Br$(E(\varepsilon )) _{p}$ are $\mathbb
Z[\mathcal{G}(E(\varepsilon )/E)]$-submodules too. Denote by $Y$ any
of the groups $E(\varepsilon ) ^{\ast }/E(\varepsilon ) ^{\ast p}$,
$_{p}{\rm Br}(E(\varepsilon ))$ and $H _{p}(E(\varepsilon )) = {\rm
Br}(E(p)(\varepsilon )/E(\varepsilon )) \cap _{p}{\rm
Br}(E(\varepsilon ))$. The preceding observations show that $Y$ can
be viewed in a natural manner as a module over the group ring
$\mathbb F _{p} [\mathcal{G}(E(\varepsilon )/E)]$, which satisfies
the following:
\par
\medskip
(3.3) The set $Y _{j} = \{y _{j} \in Y\colon \ \varphi (y _{j}) = s
^{j}y _{j}\}$ is an $\mathbb F _{p}[\mathcal{G}(E(\varepsilon
)/E)]$-submodule of $Y$, for $j = 0, 1, \dots , m - 1$; the sum of
submodules $Y _{0}, \dots , Y _{m-1}$, is direct and equal to $Y$.
\par
\medskip\noindent
In view of the Merkur'ev-Suslin theorem \cite{MS}, (16.1), this
ensures that
\par
\medskip
(3.4) $_{p}{\rm Br}(E(\varepsilon ))$ is generated by the similarity
classes of symbol division $E(\varepsilon )$-algebras $A
_{\varepsilon }(E(\varepsilon ); a, b)$ of Schur index $p$, where
$a$ and $b$ lie in he union of the sets $E _{\varepsilon
,j} = \{a _{j} \in E(\varepsilon ) ^{\ast }\colon \ \varphi (a
_{j})a _{j} ^{-s ^{j}} \in E(\varepsilon ) ^{\ast p}\}$, $j = 0, 1,
\dots , m - 1$.
\par
\medskip\noindent
It is easily obtained from Kummer theory and elementary properties
of cyclic $E(\varepsilon )$-algebras (cf. \cite{L}, Ch. VIII, Sect.
6, and \cite{P}, Sect. 15.1, Corollary~a) that (3.4) can be
supplemented as follows:
\par
\medskip
(3.5) If $a \in E _{\varepsilon ,j'}$ and $b \in E _{\varepsilon
,j''}$, then $[A _{\varepsilon }(E(\varepsilon ); a, b)] \in
_{p}{\rm Br}(E(\varepsilon )) _{\bar j}$, where $\bar j$ is
the remainder of $j ^{\prime } + j ^{\prime \prime } - 1$ modulo
$m$.
\par
\medskip\noindent
Denote for brevity $H _{p}(E(\varepsilon )) _{1}$ by $H _{p,1}(E)$.
Using (3.4) and (3.5) (see also \cite{P}, Sect. 15.1, Proposition~b,
and for more details, \cite{Ch6}, Sect. 3), one concludes that $\mathcal{G}(E(p)/E)$ is a free pro-$p$-group if and only if any of
the following three equivalent conditions holds:
\par
\medskip
(3.6) (i) $H _{p,1}(E) = \{0\}$;
\par
(ii) $E _{\varepsilon } \subseteq N(L(\varepsilon )/E(\varepsilon
))$, for each finite extension $L$ of $E$ in $E(p)$;
\par
(iii) $L _{\varepsilon } \subseteq N(L ^{\prime }(\varepsilon
)/L(\varepsilon ))$ whenever $L ^{\prime }$ is a finite extension of
$E$ in $E(p)$, $L \in I(L ^{\prime }/E)$ and $L$ is a maximal
subfield of $L ^{\prime }$.
\par
\medskip
It is known that the class of Demushkin groups is closed under
taking open subgroups (see, e.g., \cite{S1}, Ch. I, 4.5). Our next
result characterizes Demushkin groups among finitely-generated
Galois groups of maximal $p$-extensions (for a proof, see
\cite{Ch6}); characterizations of Demushkin groups in the class of
finitely generated one-relator pro-$p$-groups can be found in
\cite{DL}.

\smallskip
\begin{prop}
Let $E$ be a field such that $r(p)_{E} \in \mathbb N$, for some $p \in
P(E)$, $p \neq {\rm char}(E)$, and let $\varepsilon $ be a primitive
$p$-th root of unity in $E _{\rm sep}$. Then the following
conditions are equivalent:
\par
{\rm (i)} $\mathcal{G}(E(p)/E)$ is a Demushkin group;
\par
{\rm (ii)} $H _{p,1}(E)$ is of order $p$ and $H _{p,1}(E) \subseteq
{\rm Br}(L(\varepsilon )/E(\varepsilon ))$, provided $L \in
I(E(p)/E)$ and $[L\colon E] = p$;
\par
{\rm (iii)} For each $\alpha \in E _{\varepsilon } \setminus
E(\varepsilon ) ^{\ast p}$ and $\Delta _{p} \in d(E(\varepsilon ))$
with $[\Delta _{p}] \in H _{p,1}(E)$, $\Delta _{p}$ is
$E(\varepsilon )$-isomorphic to $A _{\varepsilon } (E(\varepsilon );
\alpha , \beta )$, for some $\beta \in E _{\varepsilon }$);
\par
{\rm (iv)} $r(p)_{R} = 2 + [R\colon E](r(p)_{E} - 2)$ in case $R
\in I(E(p)/E)$ and $[R\colon E] \in \mathbb N$.
\end{prop}
\par
\medskip
Note that the implication (i)$\to $(iv) in Proposition 3.4 follows
from \cite{Ko1}, Proposition~5.4, Galois theory and the fact that
every Demushkin group $P$ is of cohomological dimension cd$(P) = 2$
(cf. \cite{S1}, Ch. I, 4.5).
\par
\medskip
\begin{rema}
Let $P _{1}$ and $P _{2}$ be pro-$p$-groups, for a given $p \in
\mathbb P$, such that $P _{j} \cong \mathbb Z _{p} ^{m _{j}} \rtimes
\Phi _{1}$, $j = 1, 2$, where $m _{j}$ is an integer $\ge 0$ and
$\Phi _{j}$ is a free pro-$p$-group with $r(\Phi _{j}) \ge 2$ or a
Demushkin group with $r(\Phi _{j}) \ge 3$, for each index $j$. It
follows from (1.1), Corollary 3.3 and Proposition 3.4 (iii) that if
$R _{j}$ is a closed proper subgroup of $\Phi _{j}$ and $r(R _{j}) =
2$, then the index of $R _{j}$ in $\Phi _{j}$ is infinite. Hence, by
\cite{Lab1}, Theorem~2 (ii), $R _{j}$ is a free pro-$p$-group. This,
combined with Corollary 3.3 and Galois cohomology (see \cite{S1},
Ch. I, 4.1), implies that if $A _{j} \le \Phi _{j}$, $A _{j} \neq
\{0\}$ and $A _{j}$ is abelian and closed in $\Phi _{j}$, then $A
_{j} \cong \mathbb Z _{p}$. Since the set of closed normal subgroups
of $\Phi _{j}$ is closed under taking centralizers, and the
automorphism group Aut$(\mathbb Z _{p})$ is isomorphic to the direct
sum $\mathbb Z/(p - 1)\mathbb Z \oplus \mathbb Z _{p}$, it is also
clear that $A _{j}$ is not normal in $\Phi _{j}$. These observations
indicate that if $P _{1} \cong P _{2}$, then $\Phi _{1} \cong \Phi
_{2}$ and $m _{1} = m _{2}$.
\end{rema}

\medskip
\section{\bf Proofs of Theorems 1.1 and 1.2 in the case of $v(K) =
G(K)$}

The main purpose of this Section is to prove Theorems 1.1 and 1.2 in
the special case where $v(K) = G(K) \neq pv(K)$. Let $(K, v)$ be a
$p$-Henselian field with char$(\widehat K) = p$. First we show that
$\widehat K$ is perfect, provided that $r(p)_{K} \in \mathbb N$. This
result is presented by the following lemma (proved in \cite{Pop},
(2.7), and \cite{E1}, Proposition~3.1, under heavier assumptions
like the one that the group $K ^{\ast }/K ^{\ast p}$ is finite).
This lemma does not require that $v(K) \neq pv(K)$.

\medskip
\begin{lemm}
Assume that $(K, v)$ is a $p$-Henselian field, such that $\widehat
K$ is imperfect and {\rm char}$(\widehat K) = p$. Then $r(p)_{K} =
\infty $. Moreover, if {\rm char}$(K) = p$ or $v(p) \in pv(K)$, then
there exists $\Lambda \in I(K(p)/K)$, such that $[\Lambda
\colon K] = [\widehat \Lambda \colon \widehat K] = p$ and $\widehat
\Lambda $ is purely inseparable over $\widehat K$.
\end{lemm}

\begin{proof}
Let $\widetilde {\mathbb F}$ be the prime subfield of $\widehat K$,
$\widetilde B$ a basis of the field $\widehat K ^{p} = \{\hat u
^{p}\colon \ \hat u \in \widehat K\}$ as a vector space over
$\widetilde {\mathbb F}$, and $B$ a (full) system of preimages of
the elements of $\widetilde B$ in $O _{v}(K)$. The condition on
$\widehat K$ guarantees that $\widetilde B$ is infinite. Suppose
first that char$(K) = p$, put $\rho (K) = \{\alpha ^{p} - \alpha
\colon \ \alpha \in K\}$, denote by $\mathbb F$ the prime subfield
of $K$, and fix a nonzero element $\pi \in M _{v}(K)$. Clearly, the
natural homomorphism of $O _{v}(K)$ upon $\widehat K$ induces an
isomorphism $\mathbb F \cong \widetilde {\mathbb F}$. Note also that
$\rho (K)$ is an additive subgroup of $K$, and by the Artin-Schreier
theorem (cf. \cite{L}, Ch. VIII, Sect. 6), the group $K/\rho (K)$
can be canonically viewed as an $\mathbb F$-vector space of
dimension $r(p)_{K}$. It follows from the $p$-Henselian property of
$K$ and the Artin-Schreier theorem that, for each $a \in O _{v}(K)$
with $v(a) = 0$ and a residue class $\hat a \notin K ^{p}$, the root
field $\Lambda _{a}$ of the polynomial $X ^{p} - X - \pi ^{-p}a$ is
a cyclic extension of $K$, such that $[\Lambda _{a}\colon K] =
[\widehat \Lambda _{a}\colon \widehat K] = p$ and $\widehat \Lambda
_{a}/\widehat K$ is purely inseparable. This implies that the cosets
$ab\pi ^{-p} + \rho (K)$, $b \in B$, are linearly independent over
$\mathbb F$, which proves Lemma 4.1 in case char$(K) = p$.
\par
Assume now that char$(K) = 0$, $\varepsilon $ is a primitive $p$-th
root of unity in $K _{\rm sep}$, and $[K(\varepsilon )\colon K] = m$.
Our first objective is to obtain the following reduction:
\par
\medskip
(4.1) It suffices for the proof of Lemma 4.1 to consider the special
case where $v(K)$ is Archimedean, $v(p) \in pv(K)$ and $v$ is
Henselian.
\par
\medskip
\noindent It is well-known that $\prod _{j=1} ^{p-1} (1 -
\varepsilon ^{j}) = p$ and $\prod _{j'=0} ^{p-1} (\varepsilon _{1} -
\varepsilon  ^{j'}) = \varepsilon - 1$, where $\varepsilon _{1} \in
K _{\rm  sep}$ is a $p$-th root of $\varepsilon $. As char$(\widehat
K) = p$, this implies that $pv ^{\prime }(\varepsilon _{1} - 1) = v
^{\prime }(\varepsilon - 1)$ and $(p - 1)v ^{\prime }(\varepsilon -
1) = v(p)$, for every valuation $v ^{\prime }$ of $K(\varepsilon
_{1})$ extending $v$. Since cyclotomic field extensions are abelian
and $m \mid (p - 1)$ (cf. \cite{L}, Ch. VIII, Sect. 3), these
results enable one to deduce from Galois theory and Ostrowski's
theorem that $K(p) \cap K(\varepsilon _{2})$ possesses a subfield
$L$, such that $[L\colon K] = p$ and $v(p) \in pv(L)$. Hence, by
Lemma 3.2, one may assume for the proof of Lemma 4.1 that $v(p) \in
pv(K)$. Replacing $(K, v)$ by $(K _{G(K)}, \hat v _{G(K)})$ and
applying (2.1) and Proposition 2.1, one sees that our considerations
further reduce to the special case where $G(K) = v(K)$. In this
case, it is easily deduced from Zorn's lemma that Is$_{v}(E)$ has a
maximal element $\Omega $ with respect to set-theoretic inclusion.
Observing that char$(K _{\Omega }) = p$, one obtains from (2.1) that
$\widehat K$ is perfect, provided that $K _{\Omega }$ is of the same
kind. Thus it turns out that it is sufficient to prove Lemma 4.1 in
the special case where $v(K)$ is Archimedean. Using the
Grunwald-Wang theorem, one arrives at the conclusion that every
cyclic extension $L$ of $K _{v}$ of degree $p$ is $K$-isomorphic to
$\widetilde L \otimes _{K} K _{v}$, for some $\widetilde L \in
I(K(p)/K)$ with $[\widetilde L\colon K] = p$. In view of Galois
theory, the subnormality of proper subgroups of finite $p$-groups,
and the $p$-Henselian property of $v$, this implies that $K _{v}(p)$
is $K$-isomorphic to $K(p) \otimes _{K} K _{v}$, whence
$\mathcal{G}(K _{v}(p)/K _{v}) \cong \mathcal{G}(K(p)/K)$.
Therefore, one may assume for the proof of Lemma 4.1 that $K = K
_{v}$. As $v(K) \le \mathbb R$, this ensures that $v$ is Henselian
(see, e.g., \cite{L}, Ch. XII), which yields (4.1).
\par
It remains for us to prove Lemma 4.1 in the case pointed out by
(4.1). For each $a \in O _{k}$ with $v(a) = 0$ and $\hat a \notin
\widehat K ^{p}$, put $\rho (a) = \prod _{j=0} ^{m-1} \varphi ^{j}
(1 + (\varepsilon - 1) ^{p}\pi ^{-p}a) ^{l ^{j}}$, where $\pi $ is
an element of $K ^{\ast }$ of value $v(\pi ) = p ^{-1}v(p)$. It is
easily verified that $\rho (a) \in K _{\varepsilon }$. Since $m \mid
(p - 1)$ and $(p - 1)v ^{\prime }(\varepsilon - 1) = v(p)$, one also
obtains by direct calculations that $v ^{\prime }(\rho (a) - 1) = v
^{\prime }(\varepsilon - 1)$. Similarly, it is proved that the
polynomial $h _{a}(X) = \pi ^{p}(\varepsilon - 1) ^{-p}.g
_{a}((\varepsilon - 1)\pi ^{-1}X)$, where $g _{a}(X) = (X + 1) ^{p}
- \rho (a)$, lies in $O _{v'}(K(\varepsilon ))[X]$ and is congruent
to $X ^{p} - ma$ modulo $M _{v'}(K(\varepsilon ))[X]$. As $v(m ^{p}
- m) \ge v(p)$ and $\tilde a \notin \widehat K$, these calculations
show that $h _{a}$ and $g _{a}$ are irreducible over $K(\varepsilon
)$, whence $\rho (a) \notin K(\varepsilon ) ^{\ast p}$. In view of
the definition of $B$, they also lead to the following conclusion:
\par
\medskip
(4.2) The co-sets $\rho (ab)K(\varepsilon ) ^{\ast p}$, $b \in B$,
are linearly independent over $\mathbb F _{p}$.
\par
\medskip\noindent
Hence, by Albert's theorem, the extension $L _{a} ^{\prime }$ of
$K(\varepsilon )$ in $K _{\rm sep}$ obtained by adjunction of a
$p$-th root of $\rho (a)$, equals $L _{a}(\varepsilon )$, for some
$L _{a} \in I(K(p)/K)$ with $[L _{a}\colon K] = p$ and $\hat a
\notin \widehat L _{a} ^{\ast p}$. Furthermore, it follows from
(4.2) that the fields $L _{ab}$, $b \in B$, are pairwise distinct,
so the equality $r(p)_{K} = \infty $ becomes an immediate
consequence of Galois theory. Lemma 4.1 is proved.
\end{proof}

\medskip
\begin{rema}
Suppose that $(K, v)$ is a valued field with char$(K) = p$ and $v(K)
\neq pv(K)$, $\rho (K) = \{\alpha ^{p} - \alpha \colon \ \alpha \in
K\}$, and $\pi _{p}$ is an element of $K ^{\ast }$, such that $v(\pi
_{p}) > 0$ and $v(\pi _{p}) \notin pv(K)$. Clearly, the co-sets $\pi
_{p} ^{-(1+p\nu )} + \rho (K)$, $\nu \in \mathbb N$, are linearly
independent over the prime subfield of $K$. Therefore, $K/\rho (K)$
is infinite, so it follows from the Artin-Schreier theorem and
Galois theory that $r(p)_{K} = \infty $ and the polynomials $f
_{n}(X) = X ^{p} - X - \pi _{p} ^{-(1+pn)}$, $n \in \mathbb N$, are
irreducible over $K$. It also turns out that the root field $L _{n}
\in I(K _{\rm sep}/K)$ of $f _{n}$ is a totally ramified extension
of $K$ in $K(p)$ of degree $p$, for each index $n$, and $L _{n'}
\neq L _{n''}$, $n ^{\prime } \neq n ^{\prime \prime }$.
\end{rema}

The main result of this Section is contained in the following lemma.

\medskip
\begin{lemm}
Let $(K, v)$ be a $p$-Henselian field, such that {\rm char}$(K) =
0$, {\rm char}$(\widehat K) = p$ and $v(K) = G(K) \neq pG(K)$. Then
$r(p)_{K} \in \mathbb N$ if and only if $G(K)$ is cyclic and $\widehat K$
is finite. When this holds, $\mathcal{G}(K(p)/K)$ is a (standardly
admissible) Demushkin group or a free pro-$p$-group depending on
whether or not $K _{h(v)}$ contains a primitive $p$-th root of
unity.
\end{lemm}

\begin{proof}
Suppose first that $v(E)$ is Archimedean. Using Grunwald-Wang's
theorem as in the proof of (4.1), one obtains that
$\mathcal{G}(K(p)/K) \cong \mathcal{G}(K _{v}(p)/K _{v})$, which
reduces our considerations to the special case where $v$ is
Henselian. Then the latter assertion of the lemma and the
sufficiency part of the former one follow from (2.6), (1.1) and
(1.2). We show that $r(p)_{K} = \infty $, provided $v(K)$ is
noncyclic or $\widehat K$ is infinite. Our argument relies on the
fact that, in the former case, $pv(K)$ is a dense subgroup of
$v(K)$. Fix a primitive $p$-th root of unity $\varepsilon \in K
_{\rm sep}$, and take $m$, $\varphi $, $s$ and $l$ as in Lemma 3.2.
Consider the sequence $\tilde \alpha _{n} = \prod _{i=0} ^{m-1}
\varphi ^{i}(1 + (\varepsilon - 1)\alpha _{n}) ^{l ^{i}}$, $n \in
\mathbb N$, satisfying the following:
\par
\medskip
(4.3) (i) If $\widehat K$ is infinite and $v(p) \notin pv(K)$, then
$\alpha _{n}$, $n \in \mathbb N$, is a system of representatives in $O
_{v}(K)$ of a subset of $\widehat K$, which is linearly independent
over the prime subfield of $\widehat K$;
\par
(ii) If $v(p) \in pv(K)$ and $\widehat K$ is infinite, then $\alpha
_{n} = \pi ^{-1}\alpha _{n} ^{\prime }$, for each $n \in \mathbb N$,
where $\pi \in K$ is chosen so that $0 < v(\pi ) \le p ^{-1}v(p)$
and $v(\pi ) \notin pv(K)$, and the sequence $\alpha _{n} ^{\prime
}$, $n \in \mathbb N$, is defined as $\alpha _{n}$, $n \in \mathbb
N$, in case (i);
\par
(iii) When $v(K)$ is noncyclic, $\alpha _{n} \in K$, $0 < v(\alpha
_{n}) \le p ^{-1}v(p)$ and $v(\alpha _{n}) \notin pv(K)$, for every
index $n$; also, $v(\alpha _{n})$, $n \in \mathbb N$, are pairwise
distinct.
\par
\medskip\noindent
It follows from (4.3) that $\tilde \alpha _{n} \in K _{\varepsilon
}$, $v _{K(\varepsilon )}(\tilde \alpha _{n}) = 0$ and $v
_{K(\varepsilon )}(\tilde \alpha _{n} - 1) \notin pv(K(\varepsilon
))$, for each $n \in \mathbb N$. Furthermore, (4.3) ensures that the
co-sets $\tilde \alpha _{n}K(\varepsilon ) ^{\ast p}$, $n \in
\mathbb N$, are linearly independent over $\mathbb F _{p}$. These
observations, combined with Albert's theorem and Ostrowski's
theorem, prove that $K$ admits infinitely many totally ramified
extensions of degree $p$ (in $K(p)$). Hence, by Galois theory,
$r(p)_{K} = \infty $, as claimed.
\par
It remains to be proved that $v(K)$ is Archimedean, provided
$r(p)_{K} \in \mathbb N$. The equality $v(K) = G(K)$ means that
$v(p) \notin H$, for any $H \in {\rm Is}_{v}(K)$. This implies that
Is$_{v}(K)$ satisfies the conditions of Zorn's lemma, whence it
contains a minimal element, say, $\overline H$, with respect to
inclusion. We prove that $v _{\overline H}(K) \neq pv _{\overline
H}(K)$ by assuming the opposite. As $v _{\overline H}(K) =
v(K)/\overline H$ and $v(K) \neq pv(K)$, this requires that
$\overline H \neq p\overline H$. Observing also that char$(K
_{\overline H}) = p$, one obtains from Proposition 2.1 and Remark
4.2 that I$(K _{\overline H}(p)/K _{\overline H})$ has a subset
$\{\widetilde K _{n}\colon \ n \in \mathbb N\}$ of totally
ramified extensions of $K _{\overline H}$ (relative to $\hat v
_{\overline H}$) of degree $p$. By (2.4) (i), the inertial lifts $K
_{n}$ of $\widetilde K _{n}$, $n \in \mathbb N$, over $K$ relative
to $v _{\overline H}$, form an infinite subset of $I(K(p)/K)$. In
view of Galois theory, however, our conclusion contradicts the
assumption that $r(p)_{K} \in \mathbb N$, so it follows that
$\overline H = p\overline H$ and $v _{\overline H}(K) \neq pv
_{\overline H}(K)$. It remains to be seen that $\overline H =
\{0\}$. Suppose that $\overline H \neq \{0\}$. Then $\hat v
_{\overline H}$ must be nontrivial, which implies that $K _{H}$ is
infinite. Since $v _{\overline H}$ is $p$-Henselian, char$(K
_{\overline H}) = p$ and $v _{\overline H}(K) \le \mathbb R$, this
leads, by the already proved special case of Lemma 4.3, to the
conclusion that $r(p)_{K} = \infty $. The obtained results shows that
$\overline H = \{0\}$, i.e. $v(K)$ is Archimedean, which completes
our proof.
\end{proof}

\medskip
\section{\bf $p$-divisible value groups}

\medskip
In this Section we prove Theorem 1.1 (i) in the case where $v(K) =
pv(K)$. The corresponding result can be stated as follows:

\medskip
\begin{prop}
Let $(K, v)$ be a $p$-Henselian field, such that {\rm
char}$(\widehat K) = p$ and $v(K) = pv(K)$. Suppose further that $p
\in P(K)$ and $r(p)_{K} \in \mathbb N$. Then $\widehat K$ is perfect
and $\mathcal{G}(K(p)/K)$ is a free pro-$p$-group.
\end{prop}

\medskip
Proposition 5.1 generalizes \cite{E1}, Proposition~3.4, which covers
the case where $K$ contains a primitive $p$-th root of unity. Note
that the assumption on $r(p)_{K}$ is essential. Indeed, it follows
from \cite{TY}, Theorem~4.1, and Ostrowski's theorem that there
exists a Henselian field $(F, w)$, such that char$(\widehat
F) = p$, $\widehat F$ is algebraically closed, $F$ contains a
primitive $p$-th root of unity, and Br$(F) _{p} \neq \{0\}$. Since
$_{p}{\rm Br}(F) \cong H ^{2}(\mathcal{G}(F(p)/F), \mathbb F _{p})$
\cite{Ta}, page 265, this means that $\mathcal{G}(F(p)/F)$ is not a
free pro-$p$-group (see \cite{S2}, Ch. I, 4.2).

\medskip
{\it Proof of Proposition 5.1.} As noted in the Introduction, it is
known that $\mathcal{G}(E(p)/E)$ is a free pro-$p$-group, for every
field $E$ of characteristic $p$. Therefore, one may assume for the
proof that char$(K) = 0$, whence the prime subfield of $K$ may be
identified with the field $\mathbb Q$ of rational numbers. First we
show that our proof reduces to the special case where $v(K)$ is
Archimedean. Observe that $H = pH$ and $v _{H}(K) = pv _{H}(K)$, for
each $H \in {\rm Is}_{v}(K)$. This follows from the equality $v(K) =
pv(K)$ and the fact that $H$ is a pure subgroup of $v(K)$. Hence, by
Ostrowski's theorem, finite extensions of $K$ in $K(p)$ are inertial
relative to $G(K)$. It is therefore clear from (2.3) (ii) that
$\mathcal{G}(K(p)/K) \cong \mathcal{G}(K _{G(K)}(p)/K _{G(K)})$.
Thus the proof of Proposition 5.1 reduces to the special case of
$v(K) = G(K)$. This implies that $v(K) \neq H(K)$, where $H(K)$ is
defined as in (2.1) (iii), so the preceding observations indicate
that it suffices for our proof to consider the special case in which
$v(K)$ is Archimedean. Applying now the Grunwald-Wang theorem
(repeatedly, as in the proof of Lemma 4.3), one obtains that
$\mathcal{G}(K(p)/K) \cong \mathcal{G}(K _{h(v)}(p)/K _{h(v)})$,
which allows us to assume further that $v$ is Henselian and $v(K)
\le \mathbb R$.
\par
For the rest of the proof, fix $\varepsilon $, $\varphi $, $s$ and
$K _{\varepsilon }$ as in Lemma 3.2, take $m$ and $l$ as in its
proof, and put $v ^{\prime } = v _{K(\varepsilon )}$. The following
statement is easily deduced from Albert's theorem, (2.6) and the
$p$-Henselian property of $v$:
\par
\medskip
(5.1) An element $\theta \in K _{\varepsilon }$ lies in
$K(\varepsilon ) ^{\ast p}$ if and only if $\theta \in K(\varepsilon
) _{v'} ^{\ast p}$. In particular, this holds when $v ^{\prime }
(\theta - 1) > r$ and $r \in v(K)$ is sufficiently large.
\par
\medskip\noindent
Since $\widehat K$ is perfect, $v(K)$ is Archimedean and $v(K) =
pv(K)$, one deduces the following statement from the density of
$K(\varepsilon )$ in $K(\varepsilon )_{v'}$:
\par
\medskip
(5.2) For each $\alpha \in K _{\varepsilon }$, there exist elements
$\alpha _{1} \in K _{\varepsilon }$ and $\tilde \alpha \in
K(\varepsilon ) ^{\ast }$, such that $\alpha _{1}\tilde \alpha ^{p}
= \alpha $ and $v ^{\prime }(\alpha _{1} - 1) > 0$.
\par
\medskip\noindent
Our objective is to specify (5.2) as follows:
\par
\medskip
(5.3) (i) For each $\alpha \in K _{\varepsilon } \setminus
K(\varepsilon ) ^{\ast p}$, there exists a sequence $\alpha _{n} \in
K _{\varepsilon }$, $n \in \mathbb N$, such that $\alpha _{n}\alpha
^{-1} \in K(\varepsilon ) ^{p}$ and $v _{1}(\alpha _{n} - 1) >
v(p)-n ^{-1}$.
\par
(ii) $K _{\varepsilon } \subseteq N(L(\varepsilon )/K(\varepsilon
))$, for every extension $L$ of $K$ in $K(p)$ of degree $p$.
\par
\medskip
First we show that (5.3) implies $\mathcal{G}(K(p)/K)$ is a free
pro-$p$-group. In view of Corollary 3.3, it is sufficient to
establish that $r(p)_{M} = 1 + p ^{k}(r(p) - 1)$ whenever $k \in
\mathbb N$, $M \in I(K(p)/K)$ and $[M\colon K] = p ^{k}$. Given such
$k$ and $M$, there exists $M _{1} \in I(M/K)$ with $[M _{1}\colon K]
= p$. Using the fact that $M _{1}(p) = K(p)$, applying Lemma 3.2 to
$M _{1}/K$, and (5.3) (i) to $M _{\varepsilon }$ and $v _{M}$, and
proceeding by induction on $k$, one obtains that $r(p)_{M} = 1 + p
^{k}(r(p)_{K} - 1)$, as required. To prove that (5.3) (i)$\to $(5.3)
(ii) we need the following fact:
\par
\medskip
(5.4) An element $z \in K(\varepsilon )$ lies in $K(\varepsilon )
_{v'} ^{\ast p}$ whenever $v ^{\prime }(z - 1) > pv ^{\prime
}(\varepsilon - 1)$.
\par
\medskip\noindent
Statement (5.4) is implied by (2.2), applied to the polynomial
$\tilde h(X) = h((\varepsilon - 1) ^{-1}X)$, where $h(X) = (X + 1)
^{p} - z$. Fix an element $\delta \in v(K)$, $\delta > 0$, and
consider the polynomials $f _{\delta }(X) = X ^{p} - 1 -
(\varepsilon - 1) ^{p}\pi _{\delta } ^{-p}t _{\delta }$, $h _{\delta
}(X) = f _{\delta }(X + 1)$, and $\tilde h _{\delta }(X) = h
_{\delta }(\pi _{\delta }(\varepsilon - 1) ^{-1}X)$, where $\pi
_{\delta } \in K(\varepsilon ) ^{\ast }$, $t _{\delta } \in K
_{\varepsilon } \in \nabla _{0}(K(\varepsilon ))$ and $0 < v
^{\prime }(\pi _{\delta }) < (1/p)v ^{\prime }(\varepsilon - 1) +
\delta /p$. Clearly, $f _{\delta }$, $h _{\delta }$ and $\tilde h
_{\delta }$ share a common root field $Y _{\delta }$ over
$K(\varepsilon )$, which equals $\widetilde Y _{\delta
}K(\varepsilon )$, for some $\widetilde Y _{\delta } \in I(K(p)/K)$
with $[\widetilde Y\colon K] = p$. In addition, it is verified by
direct calculations that if $\tau _{\delta } \in K _{\rm sep}$ is a
zero of $\tilde h$, then $v ^{\prime }(\tilde h ^{\prime }(\tau
_{\delta })) = v ^{\prime }(\pi _{\delta } ^{p-1})$. This, combined
with (2.2), (5.1) and (5.4), shows that $N(Y _{\delta
}/K(\varepsilon ))$ includes the set $K(\varepsilon ) \cap \nabla
_{\delta '}(K(\varepsilon ))$, where $\delta ^{\prime } = (2p - 2)p
^{-1}(v ^{\prime }(\varepsilon - 1) + \delta )$. In view of Albert's
theorem, the obtained result, applied to any sufficiently small
$\delta $, yields (5.3) (i)$\to $(5.3) (ii).
\par
It remains for us to prove (5.3) (i). First we consider the special
case in which $\widehat K$ contains more than $1 + \sum _{j=0}
^{m-1} l ^{j} \colon = \bar l$ elements. Fix a sequence $\hat \alpha
_{\nu }$, $\nu = 1, \dots \bar l$, of pairwise distinct elements of
$\widehat K ^{\ast }$, and let $\alpha _{\nu }$ be a preimage of
$\hat \alpha _{\nu }$ in $O _{v}(K)$, for each index $\nu $. Assume
that (5.3) (i) is false, i.e. there exists $d \in K _{\varepsilon
}$, for which the supremum $c(d)$ of values $v ^{\prime }(d ^{\prime
} - 1)$, $d ^{\prime } \in dK(\varepsilon ) ^{\ast p}$, is less than
$v(p)$. Take $d$ so that $c(d)$ is maximal with respect to this
property, and put $\bar \delta = p ^{-1-\bar l}(v(p) - c(d))$, and
take $d ^{\prime } \in K _{\varepsilon }$ so that $v ^{\prime }(d
^{\prime } - 1) > c(d) - \bar \delta $. Denote by $p _{d'}$ the
polynomial $p _{d'}(X) = \prod _{i=0} ^{m-1} (1 + \varphi ^{s
^{i}}(\pi ^{\prime })X) ^{l ^{i}}$, where $\pi ^{\prime } = d
^{\prime } - 1$, and put $p _{d'}(X) = 1 + \sum _{\nu =1} ^{\bar l}a
_{\nu }X ^{\nu }$, $s _{d'} = \{a _{\nu }\colon \ 1 \le \nu \le \bar
l, v ^{\prime }(a _{\nu }) < v(p)\}$. It is easily verified that
$f(\lambda ) \in K _{\varepsilon }$, for every $\lambda \in M
_{v}(K)$. Now fix an element $\pi \in M _{v}(K)$ so that $c(d) <
v(\pi ^{\prime }\pi ^{p})$ and $v(\pi ^{p}) < \bar \delta $. It
follows from the choice of $d$ and $\pi $ that there exist $b _{1},
\dots , b _{\bar l} \in K(\varepsilon )$ satisfying the following
conditions, for every index $\nu $:
\par
\medskip\noindent
(5.5) $\ b _{\nu }f((\alpha _{\nu }\pi ) ^{p}) ^{-1} \in
K(\varepsilon ) ^{\ast p}$ and $v ^{\prime }(b _{\nu } - 1) \ge v(p)
- \bar \delta $.
\par
\medskip\noindent
More precisely, one can find elements $\tilde b _{1}, \dots , \tilde
b _{\bar l} \in \nabla _{0}(K(\varepsilon ))$ so that $\tilde b
_{\nu } ^{p} = b _{\nu }f((\alpha _{\nu }\pi ) ^{p}) ^{-1}$ and $b
_{\nu }$ satisfies (5.5), for $\nu = 1, \dots , \bar l$. This means
that $K(\varepsilon )$ contains elements $\pi _{1}, \dots , \pi
_{\bar l}$, such that $v ^{\prime }(f((\alpha _{\nu }\pi ) ^{p}) - 1
- \pi _{\nu } ^{p}) \ge v(p) - \bar \delta $, $\nu = 1, \dots , \bar
l$. As $f((\alpha _{\nu }\pi ) ^{p}) - 1 = \sum _{\mu =1} ^{\bar l}
\alpha _{\nu } ^{p\mu }(\pi ^{p}a _{\mu }) ^{\mu }$, for each index
$\nu $, the obtained result and the choice of the sequence $a _{\nu
}$, $\nu = 1, \dots , \bar l$, enables one to deduce from basic
linear algebra that there exist $\tilde \pi _{1}, \dots , \tilde \pi
_{\bar l} \in K(\varepsilon )$, for which $v ^{\prime }(a _{\mu }\pi
^{p\mu } - \tilde \pi _{\mu } ^{p}) \ge v(p) - \bar \delta $, $\mu =
1, \dots , \bar l$. Thus it turns out that $v ^{\prime }(a _{\mu } -
(\tilde \pi _{\mu }.\pi ^{-\mu }) ^{p}) \ge v(p) - \bar \delta -
p\mu v(\pi )$, for every index $\mu $. Hence, there exists $\tilde
\pi \in K(\varepsilon )$, such that $v ^{\prime }(f(1) - 1 - \tilde
\pi ^{p}) = v ^{\prime }(-1 + f(1)(1 + \tilde \pi ) ^{-p}) \ge v(p)
- \bar \delta - p\bar lv(\pi )$. Moreover, it follows from the
choice of $\pi $ that $v(p) - \bar \delta - p\bar lv(\pi ) > c(d)$.
Since, however, $f(1)d ^{-m} \in K(\varepsilon ) ^{\ast p}$, the
assumption on $c(d)$ and our calculations require that $c(d) \ge
v(p) - \bar \delta - p\bar lv(\pi )$. The obtained contradiction
proves (5.3) (i) when $\widehat K$ contains more than $\bar l$
elements.
\par
Suppose finally that $\widehat K$ is of order $\le \bar l$. Then $K$
has an inertial extension $K _{n}$ in $K(p)$ of degree $p ^{n}$, for
each $n \in \mathbb N$. As $\widehat K$ is perfect and $v(K) =
pv(K)$, $K ^{\ast } = K ^{\ast p ^{n}}\nabla _{0}(K)$, which implies
$N(K _{n}/K) = K ^{\ast }$. When $n$ is sufficiently large,
$\widehat K _{n}$ contains more than $\bar l$ elements, so one can
apply (5.3) (i) to $K _{n}(\varepsilon )$ and the prolongation of $v
_{K _{n}}$ on $K _{n}(\varepsilon )$. Since char$(\widehat K) = p$,
$K$ is a nonreal field, by \cite{La}, Theorem~3.16, so it follows
from Galois theory and \cite{Wh}, Theorem~2, that $G(K(p)/K)$ is a
torsion-free group. In view of Galois cohomology (see \cite{S3} and
\cite{S2}, Ch. I, 4.2), these observations show that (5.3) (i) holds
in general, which completes the proof of Proposition 5.1.

\par
\medskip
\begin{coro}
Let $(K, v)$ be a $p$-Henselian field, for a given $p \in \mathbb
P$, and suppose that $v(K) = pv(K)$, $\mathcal{G}(K(p)/K)$ is a
Demushkin group and $r(p)_{K} \ge 3$. Then $p \neq {\rm
char}(\widehat K)$.
\end{coro}

\begin{proof}
As cd$(\mathcal{G}(K(p)/K)) = 2$, $\mathcal{G}(K(p)/K)$ is not a
free pro-$p$-group (cf. \cite{S2}, Ch. I, 4.2 and 4.5), our
conclusion follows from Proposition 5.1.
\end{proof}

\smallskip
\begin{rema}
Let $(K, v)$ be a $p$-Henselian field, for some $p \in \mathbb P$,
and $V(K)$ the set of valuation subrings of $K$ included in $O
_{v}(K)$ and corresponding to $p$-Henselian valuations of $K$. Then
it can be deduced from Zorn's lemma that $K$ has a valuation $w$,
such that $O _{w}(K)$ is a minimal element of $V(K)$ with respect to
inclusion. Hence, by Proposition 2.1, the residue field $\widetilde
K _{[w]}$ of $(K, w)$ does not possess $p$-Henselian valuations.
Note also that, in the setting of Corollary 5.2,
$\mathcal{G}(K(p)/K) \cong \mathcal{G}(\widetilde K
_{[w]}(p)/\widetilde K _{[w]})$.
\end{rema}

\medskip
Theorem 1.2 and \cite{TY}, Proposition~2.2 and Theorem~3.1, ensure
that if $(K, v)$ is $p$-Henselian with $G(K) \neq pG(K)$,
char$(\widehat K) = p$ and $r(p)_{K} \in \mathbb N$, then finite
extensions of $K$ in $K(p)$ are defectless. As shown in \cite{Ch1},
Sect. 6, this is not necessarily true without the assumption on
$r(p)_{K}$ even when $\mathcal{G}(K(p)/K)$ is a Demushkin group. The
proof relies on a lemma, which is essentially equivalent to the
concluding result of this Section.

\medskip
\begin{prop}
In the setting of Proposition 5.1, suppose that $K$ contains a
primitive $p$-th root of unity $\varepsilon $, $v(K) \le \mathbb R$,
$F$ is an extension of $K$ in $K(p)$ of degree $p$, and $\psi $ is a
generator of $\mathcal{G}(F/K)$. Then there exist $\lambda _{n} \in
F$, $n \in \mathbb N$, such that $0 < v _{F}(\lambda _{n}) < v
_{L}(\psi (\lambda _{n}) - \lambda _{n}) < 1/n$, for each index $n$.
\end{prop}

\begin{proof}
It is clear from the $p$-Henselity of $v$ that if $F/K$ is
inertial, then $F/K$ has a primitive element $\xi $, such that
$v(\xi ) = v(\psi (\xi ) - \xi ) = 0$. Since $\widehat K$ is
perfect and $v(K) = pv(K)$, this allows us to assume further that
$F/K$ is immediate. Then our assertion is deduced by the method of
proving \cite{Pop}, (2.7), and \cite{Ch6}, (3.1) (see also
\cite{E1}, Lemma~3.3).
\end{proof}

\medskip

\section{\bf $p$-Henselian valuations with $p$-indivisible value
groups}

The purpose of this Section is to complete the proofs of Theorems
1.1 and 1.2. As a major step in this direction, we prove the
following lemma.

\medskip
\begin{lemm}
Let $(K, v)$ be a $p$-Henselian field with {\rm char}$(\widehat K)
\neq p$ and $v(K) \neq pv(K)$, for a given $p \in \mathbb P$, and
let $\varepsilon \in K _{\rm sep}$ be a primitive $p$-th root of
unity. Then:
\par
{\rm (i)} $\mathcal{G}(K(p)/K) \cong \mathcal{G}(\widehat
K(p)/\widehat K)$, provided that $\varepsilon \notin K _{h(v)}$;
\par
{\rm (ii)} $\mathcal{G}(K(p)/K) \cong \mathcal{G}(K _{\Delta
(v)}(p)/K _{\Delta (v)})$, if $\varepsilon \in K _{h(v)}$ and $v(K)
\neq \Delta (v)$;
\par
{\rm (iii)} If $\varepsilon \in K _{h(v)}$, then
$\mathcal{G}(K(p)/K) \cong \mathbb Z _{p} ^{\tau (p)} \rtimes
\mathcal{G}(\widehat K(p)/\widehat K)$, where $\tau (p)$ is the
dimension of $\Delta (v)/p\Delta (v)$ as an $\mathbb F _{p}$-vector
space and $\mathbb Z _{p} ^{\tau (p)}$ is a topological group
product of isomorphic copies of $\mathbb Z _{p}$, indexed by a set
of cardinality $\tau (p)$; in particular, if $\Delta (v) = p\Delta
(v)$, then $\mathcal{G}(K(p)/K) \cong \mathcal{G}(\widehat
K(p)/\widehat K)$.
\end{lemm}

\begin{proof} Suppose that $[K(\varepsilon )\colon K] = m$, take
$\sigma $ as in (2.5) and $\varphi $, $s$, $l$ as in Lemma 3.2, and
denote by $U$ the compositum of the inertial extensions of $K$ in
$K(p)$ relative to $v$. As $(K _{h(v)}, \sigma )/(K, v _{G(K)})$ is
immediate, it follows from the $p$-Henselian property of $v$ that $K
_{h(v)} \cap K(p) = K$. At the same time, by (2.5), $U.K _{h(v)}$
equals the compositum of the inertial extensions of $K _{h(v)}$ in
$K _{h(v)}(p)$ relative to $\sigma $. It is therefore clear from
\cite{Ch2}, Lemma~1.1 (a), that if $\varepsilon \notin K _{h(v)}$,
then $K _{h(v)} (p) = U.K _{h(v)} = K(p)K _{h(v)}$ and
$\mathcal{G}(K(p)/K) \cong \mathcal{G}(K _{h(v)} (p)/K _{h(v)})
\cong \mathcal{G}(\widehat K (p)/\widehat K)$. Thus Lemma 6.1 (i) is
proved.
\par
\smallskip
In the rest of the proof, we assume that $\varepsilon \in K
_{h(v)}$. Suppose first that $K(\varepsilon )$ embeds in $K _{v}$
over $K$. Fix a prolongation $v ^{\prime }$ of $v$ on $K(\varepsilon
)$, and put $v _{j} = v ^{\prime } \circ \varphi ^{j-1}$, for $j =
1, \dots , m$. Then the valuations $v _{1}, \dots , v _{m}$ are
independent, so it follows from the Approximation Theorem (cf.
\cite{B1}, Ch. VI, Sect. 7.2) that, for each $\gamma \in v(K)
\setminus pv(K)$, there exists $\alpha _{\gamma } \in K(\varepsilon
)$ of values $v _{1}(\alpha _{\gamma }) = \gamma $ and $v
_{j}(\alpha _{\gamma }) = 0$, $1 < j \le m$. Put $\tilde \alpha
_{\gamma } = \prod _{i=0} ^{m-1} \varphi ^{i}(\alpha _{\gamma }) ^{l
^{i}}$ and denote by $\widetilde K _{\tilde \alpha }$ the extension
of $K(\varepsilon )$ generated by the $p$-th roots of $\tilde \alpha
$ in $K _{\rm sep}$. It is easily verified that $\tilde \alpha
_{\gamma } \in K _{\varepsilon }$ and $v _{1}(\tilde \alpha _{\gamma
}) = \gamma $. Hence, by applying to $\widetilde K _{\tilde \alpha
}$ Ostrowski's theorem and Albert's theorem, one obtains the
following result:
\par
\medskip
(6.1) There exists $K _{\gamma } \in I(K(p)/K)$ with $[K
_{\gamma }\colon K] = p$ and $\gamma \in pv(K _{\gamma })$.
\par
\medskip\noindent As char$(\widehat K) \neq p$, it is clear from
(2.4), (6.1) and Ostrowski's theorem that every extension of $K
_{h(v)}$ in $K _{h(v)} (p)$ of degree $p$ is included in $K(p)K
_{h(v)}$. Proceeding by induction on $n$, using (6.1), and arguing
as in the proof of (2.4), one proves that $K(p)K _{h(v)}$ includes
every extension of $K _{h(v)}$ in $K _{h(v)}(p)$ of degree $p ^{n}$.
Thus the equality $K _{h(v)} (p) = K _{h(v)}K(p)$ becomes obvious as
well as the fact that $\mathcal{G}(K _{h(v)}(p)/K _{h(v)}) \cong
\mathcal{G}(K(p)/K)$. Note further that, by (2.5), $U$ does not
admit inertial proper extensions in $K _{h(v)}(p)$ and there are
canonical isomorphisms $\mathcal{G}(UK _{h(v)}/K _{h(v)}) \cong
\mathcal{G}(U/K) \cong \mathcal{G}(\widehat K(p)/\widehat K)$. Using
the fact that char$(\widehat K) \neq p$, one also obtains that
$\nabla _{0}(K _{h(v)}) \subseteq K _{h(v)} ^{\ast p ^{n}}$, for
every $n \in \mathbb N$. As $\varepsilon \in K _{h(v)}$ and $(K
_{h(v)}, \sigma )/(K, v)$ is immediate, one also sees that $\widehat
U ^{\ast } = \widehat U ^{\ast p ^{n}}$, $n \in \mathbb N$. These
observations indicate that $[(K _{\gamma }UK _{h(v)})\colon $ $UK
_{h(v)}] = p$ and the field $U _{\gamma } = K _{\gamma }UK _{h(v)}$
depends only on the subgroup of $v(K)/pv(K)$ generated by the coset
of $\gamma $. Fix a minimal system of generators $V _{p}$ of
$v(K)/pv(K)$ and a full system $W _{p}$ of representatives of the
elements of $V _{p}$ in $v(K)$. The preceding observations show that
every extension of $UK _{h(v)}$ in $K _{h(v)}(p)$ of degree $p$ is
included in the compositum of the fields $U _{\gamma }$, $\gamma \in
W _{p}$. Note further that, for each $\gamma \in W _{p}$, $U
_{\gamma }$ is generated over $UK _{h(v)}$ by a $p$-th root of an
element $u _{\gamma } \in UK _{h(v)}$ of value $\gamma $. Since
char$(\widehat K) \neq p$ and $\varepsilon \in K _{h(v)}$, it also
follows that $UK _{h(v)}$ contains a primitive $p ^{n}$-th root of
unity, for each $n \in \mathbb N$. In view of Kummer theory, this
leads to the conclusion that there exists an abelian extension $Z$
of $UK _{h(v)}$ in $K _{h(v)}(p)$ with $\sigma (Z) = p\sigma (Z)$.
Since $Z$ does not admit inertial proper extensions in $K
_{h(v)}(p)$, the obtained result, the inequality char$(\widehat K)
\neq p$ and Ostrowski's theorem imply that $Z = K _{h(v)}(p)$ and so
prove the following:
\par
\smallskip
(6.2) $K _{h(v)}(p)$ is abelian over $U.K _{h(v)}$ with
$\mathcal{G}(K _{h(v)}(p)/UK _{h(v)}) \cong \mathbb Z _{p} ^{\tau
(p)}$.
\par
\medskip\noindent
Note further that the set $\Theta (K _{h(v)}) = \{\Theta \in {\rm
I}(K _{h(v)}(p)/K _{h(v)})\colon \ UK _{h(v)} \cap \Theta = K
_{h(v)}\}$, partially ordered by inclusion, satisfies the conditions
of Zorn's lemma, whence it possesses a maximal element $T$. In view
of (6.1), $\sigma (T) = p\sigma (T)$, so it follows from Ostrowski's
theorem and the inequality char$(\widehat K) \neq p$ that $UT = K
_{h(v)}(p)$. Hence, by Galois theory and the equality $UK _{h(v)}
\cap T = K _{h(v)}$, $\mathcal{G}(K _{h(v)}(p)/K _{h(v)}) \cong
\mathcal{G}(K _{h(v)}(p)/UK _{h(v)}) \rtimes \mathcal{G}(UK
_{h(v)}/K _{h(v)})$. This, combined with (6.2) and the isomorphism
$\mathcal{G}(UK _{h(v)}/K _{h(v)}) \cong $
\par\noindent
$\mathcal{G}(\widehat K(p)/\widehat K)$, proves the assertion of
Lemma 6.1 (ii) in case $\Delta (v) = v(K)$.
\par
\smallskip
Assume finally that $\varepsilon \in K _{h(v)}$ and $\Delta (v) \neq
v(K)$, put $\delta = \hat v _{\Delta (v)}$, and denote for brevity
by $\widetilde K$ the Henselization of $K _{\Delta (v)}$ relative to
$\delta _{G(K)}$. It is clear from the definition of $\Delta (v)$
that $\varepsilon \notin K _{h(v _{\Delta (v)})}$. In view of (2.5)
with its proof, and of \cite{Ch2}, Lemma~1.1 (a), this ensures that
$K _{h(v _{\Delta (v)})}(p) = K(p).K _{h(v _{\Delta (v)})}$ and
$\mathcal{G}(K(p)/K) \cong \mathcal{G}(K _{\Delta (v)}(p)/K _{\Delta
(v)})$. At the same time, it follows from the definition of $\Delta
(v)$ that each Henselization of $K _{\Delta (v)}$ relative to $\hat
v _{\Delta (v)}$ contains a primitive $p$-th root of unity. Hence,
by Proposition 2.2, $\widetilde K$ contains such a root as well.
Note also that, by (2.1) (i), $\widehat K$ is isomorphic to the
residue field of $(K _{\Delta (v)}, \hat v _{\Delta (v)})$. These
observations, combined with (6.2), indicate that $\widetilde K(p) =
K _{\Delta (v)}(p).\widetilde K$ and $\mathcal{G}(\widetilde
K(p)/\widetilde K) \cong \mathcal{G}(K _{\Delta (v)}(p)/K _{\Delta
(v)}) \cong \mathbb Z _{p} ^{\tau (p)} \times \mathcal{G}(\widehat
K(p)/\widehat K)$, which completes the proof of Lemma 6.1.
\end{proof}

Lemma 6.1, Proposition 2.2 and our next lemma, enable one to deduce
Theorems 1.1 and 1.2 from Lemma 4.3 and Proposition 5.1.

\medskip
\begin{lemm}
Let $(K, v)$ be a $p$-Henselian field with {\rm char}$(K) = 0$, {\rm
char}$(\widehat K) = p$, $v(K) \neq pv(K)$ and $v(K) \neq G(K)$.
Take $r(p)_{K}$, $\varepsilon $, $K _{h(v)}$ and $G(K)$ as in
Theorems 1.1 and 1.2, and put $r(p) = r(p)_{K}$, $r(p) ^{\prime } =
r(p)_{K _{G(K)}}$. Then:
\par
{\rm (i)} $\mathcal{G}(K(p)/K) \cong \mathcal{G}(K _{G(K)}(p)/K
_{G(K)})$, provided that $\varepsilon \notin K _{h(v)}$;
\par
{\rm (ii)} If $\varepsilon \in K _{h(v)}$ and $\Delta (v)$ is
defined as in Theorems 1.1 and 1.2, then $r(p) \in \mathbb N$ if and
only if $r(p) ^{\prime } \in \mathbb N$ and $\Delta (v)/p\Delta (v)$
is a finite group; when this holds, $\Delta (v)/p\Delta (v)$ is of
order $p ^{r(p)-r(p)'}$;
\par
{\rm (iii)} When $r(p) ^{\prime } \in \mathbb N$ and $G(K) = pG(K)$,
$\mathcal{G}(K _{G(K)}(p)/K _{G(K)})$ is a free pro-$p$-group.
\end{lemm}

\begin{proof}
Proposition 2.1 and our assumptions show that $v _{G(K)}$ is
$p$-Henselian and char$(K _{G(K)}) = 0$. In addition, it is not
difficult to see that $\Delta (v)/G(K) = \Delta (v _{G(K)})$. This,
combined with Lemma 6.1, proves Lemma 6.2 (i) and (ii). Since, by
Proposition 2.1, $\hat v _{G(K)}$ is $p$-Henselian, Lemma 6.2 (iii)
is proved by applying Proposition 5.1 to $(K _{G(K)}, \hat v
_{G(K)})$.
\end{proof}

\medskip
\begin{rema}
Let $(K, v)$ be a $p$-Henselian field with {\rm char}$(\widehat K) =
p$, $r(p)_{K} \in \mathbb N$ and cd$(\mathcal{G}(K(p)/K) \ge 2$.
Then (6.1), Lemma 6.1 and the proof of Lemma 4.3 imply the existence
of a $p$-Henselian field $(\Lambda , z)$, such that
$\mathcal{G}(\Lambda (p)/\Lambda ) \cong \mathcal{G}(K(p)/K)$,
$\widehat \Lambda \cong \widehat K$, $z(\Lambda ) = \Delta (v)$, and
$\Lambda $ contains a primitive $p$-th root of unity. More
precisely, one may put $\Lambda = K _{h(v)}$ or take as $\Lambda $
the Henselization of $K _{\Delta (v)}$ relative to $\hat v _{\Delta
(v)}$, depending on whether or not $\Delta (v) = v(K)$. In view of
\cite{E1}, Theorems~3.7 and 3.8, this enables one to describe the
isomorphism classes of $\mathcal{G}(L(p)/L)$ when $(L, \lambda )$
runs across the class of $p$-Henselian fields singled out by
Theorems 1.1 and 1.2.
\end{rema}

\medskip
Our next result extends the scope of \cite{W2}, Lemma~7, as follows:

\medskip
\begin{coro} For a pro-$p$-group $P$ of rank $2$, the following
conditions are equivalent:
\par
{\rm (i)} $P$ is a free pro-$p$-group or a Demushkin group;
\par
{\rm (ii)} There exists a $p$-Henselian field $(K, v)$ with {\rm
char}$(\widehat K) = p$ and
\par\noindent
$\mathcal{G}(K(p)/K) \cong P$.
\end{coro}

\begin{proof}
Suppose first that $P$ satisfies (ii). Then, by Remark 6.3, $K$ can
be chosen so as to contain a primitive $p$-th root of unity. Hence,
by \cite{W2}, Lemma~7, $P$ satisfies (i) in case $p > 2$, so we
assume further that $p = 2$ and $P$ is not a free pro-$2$-group. In
view of Galois cohomology, this means that Br$(K) _{2} \neq \{0\}$,
and by Proposition 3.4, it suffices to show that, for each quadratic
extension $L/K$, $N(L/K)$ is of index $2$ in $K ^{\ast }$. As $r(P)
= 2$ and $K ^{\ast 2} \subseteq N(L/K)$, this amounts to proving
that $N(L/K)$ contains an element $a _{L} \in K ^{\ast } \setminus K
^{\ast 2}$. When $\sqrt{-1} \in K$, the assertion follows from
Kummer theory. Assume now that $-1 \notin K ^{\ast 2}$ and $M =
K(\sqrt{-1})$. Observing that, by \cite{La}, Theorem~3.16, and the
equality char$(\widehat K) = 2$, $K$ is nonreal, one obtains from
\cite{Wh}, Theorem~2, and Galois theory that $\mathcal{G}(K(p)/K)$
is torsion-free. Hence, by Galois cohomology \cite{S2},
$\mathcal{G}(K(p)/M)$ is not a free pro-$2$-group, so our argument
proves that $\mathcal{G}(K(2)/M)$ is Demushkin. It is therefore
clear from \cite{S1}, Ch. I, 4.5, that $\mathcal{G}(K(2)/K)$ is also
a Demushkin group, so (ii)$\to $(i).
\par
We prove that (i)$\to $(ii). In view of (1.1) and (1.2), applied to
$F = \mathbb Q_{p}$, $p > 2$, one may consider only the case where
$P$ is a Demushkin group. Let $\omega $ be the canonical discrete
valuation of $\mathbb Q _{p}$, $I _{p}$ the compositum of the
inertial extensions of $\mathbb Q_{p}$ in $\mathbb Q_{p}(p)$
relative to $\omega $, $\Gamma _{p} ^{\prime }$ the extension of $I
_{p}$ generated by the primitive roots of unity in $\mathbb
Q_{p,{\rm sep}}$ of $p$-primary degrees, and $\Gamma _{p}$ the
$\mathbb Z _{p}$-extension of $I _{p}$ in $\Gamma _{p} ^{\prime }$.
Clearly, $\Gamma _{p} ^{\prime }/I _{p}$ is abelian, and it follows
from (1.1) and Galois theory that $r(p)_{I _{p}} = \infty $. Note
also that $\omega _{I _{p}}$ is discrete and finite extensions of $I
_{p}$ in $\Gamma _{p} ^{\prime }$ are totally ramified relative to
$\omega _{I _{p}}$. Fix a field $R \in I(\Gamma _{p} ^{\prime }/I
_{p})$ so that $\mathcal{G}(\Gamma _{p} ^{\prime }/R)$ is a
procyclic pro-$p$-group and put $\Delta _{R} = \{R ^{\prime } \in
I(\mathbb Q _{p,{\rm sep}}/R)\colon \ R ^{\prime } \cap \Gamma _{p}
^{\prime } = R\}$. It is clear from the definition of $R$ that it
contains a primitive $p$-th root of unity $\varepsilon _{p}$, and it
follows from Galois theory that $r(p)_{R} = \infty $ and
$\mathcal{G}(\Gamma _{p} ^{\prime }/R) \cong \mathcal{G}(\Gamma _{p}
^{\prime }R ^{\prime }/R ^{\prime })$, for each $R ^{\prime } \in
\Delta _{R}$. Since $R$ is nonreal, this enables one to deduce from
Zorn's lemma, Galois theory and \cite{Wh}, Theorem~2, that there
exists $R _{p} \in \Delta _{R}$, such that $\mathcal{G}(\mathbb Q
_{p,{\rm sep}}/R _{p}) \cong \mathbb Z _{p}$. As $r(p)_{R} = \infty
$, this implies that $R _{p}$ contains as a subfield an infinite
extension of $I _{p}(\varepsilon )$. Hence, by the noted properties
of $\omega _{I _{p}}$, $\mathbb Q _{p,{\rm sep}}/R _{p}$ is
immediate relative to $\omega _{R _{p}}$. Let now $\omega (R _{p}) =
H$ and $K$ be a Laurent formal power series field in one
indeterminate over $R _{p}$. It is easy to see that then $K$ has a
Henselian valuation $v$ extending $\omega _{R _{p}}$, such that $H
\in {\rm Is}_{v}(K)$, $v(K)/H \cong \mathbb Z$ and $K _{H} \cong R
_{p}$. This, combined with Lemma 6.1 and the fact that $\varepsilon
_{p} \in R _{p}$, shows that $\mathcal{G}(K(p)/K)$ is a Demushkin
group and $r(p)_{K} = 2$. Observing finally that $R$ can be chosen
so that $\mathcal{G}(K(p)/K) \cong P$ (cf. \cite{Lab2}), one obtains
that (i)$\to $(ii), which completes our proof.
\end{proof}

The method of proving Theorems 1.1 and 1.2 has the following
application to the study of ramification properties of
$p$-extensions of $p$-Henselian fields.

\smallskip
\begin{coro}
Let $(K, v)$ be a $p$-Henselian field such that {\rm char}$(\widehat
K) = p$ and $v(K) \neq pv(K)$. Then $K$ has no totally ramified
extension in $K(p)$ of degree $p$ if and only if {\rm char}$(K) = 0$
and one of the following conditions holds:
\par
{\rm (i)} $G(K) = pG(K)$ and $K _{h(v)}$ contains no primitive
$p$-th roots of unity;
\par
{\rm (ii)} $K _{h(v)}$ contains a primitive $p$-th root of unity and
$\Delta (v) = p\Delta (v)$.
\end{coro}

\begin{proof}
In view of Remark 4.2, one may consider only the case of char$(K) =
0$. Statement (4.3) and the proof of Lemma 4.3 indicate that if
$v(K) = G(K)$, then $K$ possesses infinitely many totally ramified
extensions in $K(p)$ of degree $p$. By (6.1), the existence of a
totally ramified cyclic extension of $K$ of degree $p$ is also
guaranteed when $v(K) = \Delta (v)$. Note further that, for each $H
\in {\rm Is}_{v}(K)$, an inertial extension $L$ of $K$ in $K(p)$
relative to $v _{H}$ is totally ramified relative to $v$ if and only
if $L _{H}/K _{H}$ is totally ramified relative to $\hat v _{H}$.
Since $v _{H}$ is $p$-Henselian and, by the proof of Lemma 6.1,
finite extensions of $K$ in $K(p)$ are inertial relative to $v
_{H}$, provided that $H$ equals $\Delta (v)$ or $G(K)$ depending on
whether or not $K _{h(v)}$ contains a primitive $p$-th root of
unity, these observations imply together with (2.4) the necessity of
conditions (i) and (ii) of Corollary 6.5. They also show that, for
the proof of their sufficiency, it remains to be seen that, in case
$H = pH$, finite extensions of $K _{H}$ in $K _{H}(p)$ are inertial
relative to $\hat v _{H}$. As char$(K _{H}) = 0$, this follows at
once from Ostrowski's theorem, so Corollary 6.5 is proved.
\end{proof}

\medskip
The concluding result of this Section shows that each standardly
admissible pro-$p$-group $P$, where $p > 2$, is isomorphic to
$\mathcal{G}(E _{P}(p)/E _{P})$, for some field $E _{P}$ without a
primitive $p$-th root of unity.

\medskip
\begin{prop}
Let $P$ be a standardly admissible Demushkin pro-$p$-group, for some
$p \in \mathbb P \setminus \{2\}$, and let $m \in \mathbb N$ be a
divisor of $p - 1$. Then there exists a field $E$, such that
$\mathcal{G}(E(p)/E) \cong P$ and $[E(\varepsilon )\colon E] = m$,
where $\varepsilon $ is a primitive $p$-th root of unity in $E _{\rm
sep}$.
\end{prop}

\begin{proof}
Let $\mathbb Q$ be the field of rational numbers, $\Gamma $ the
unique $\mathbb Z _{p}$-extension of $\mathbb Q$ in $\mathbb Q _{\rm
sep}$, and for each $n \in \mathbb N$, let $\Gamma _{n}$ be the
subextension of $\mathbb Q$ in $\Gamma $ of degree $p ^{n-1}$. Fix a
primitive $p$-th root of unity $\varepsilon \in \mathbb Q _{\rm
sep}$, denote by $R _{m}$ the extension of $\mathbb Q$ in $\mathbb
Q(\varepsilon )$ of degree $(p - 1)/m$, suppose that the reduced
component of $C(P)$ is of order $p ^{\nu }$ (see (1.4)), put $\Phi =
R _{m}\Gamma _{\nu }$, and let $\omega $ be a valuation of $\Phi $
extending the normalized $p$-adic valuation of $\mathbb Q$. It
follows from the Grunwald-Wang theorem that there is a cyclic
extension $F$ of $\Phi $ in $\mathbb Q _{\rm sep}$, such that
$[F\colon \Phi ] = m$, $F \cap \mathbb Q(\varepsilon ) = \Phi $ and
$\Phi (\varepsilon )$ is embeddable over $\Phi $ in $F _{f}$, where
$f$ is a valuation of $F$ extending $\omega $. Using the same
theorem, one proves the existence of a field $F ^{\prime } \in
I(\mathbb Q _{\rm sep}/F)$, such that $F ^{\prime } \cap
F(\varepsilon ) = F$, $F ^{\prime }/F$ is inertial relative to $f$
and $[F ^{\prime }\colon F] = r$, where $r = (r(P)-2)/(p ^{\nu }-p
^{\nu -1})$. This implies that $F ^{\prime }$ has a unique, up-to an
equivalence, valuation $f ^{\prime }$ extending $f$. Consider now
some $p$-Henselization $E \in I(F ^{\prime }(p)/F ^{\prime })$ of
$F$ relative to $f ^{\prime }$. It follows from Galois theory and
the definition of $E$ that $[E(\varepsilon )\colon E] = m$.
Observing that $\varepsilon \in E _{h(f)}$, $[E _{f}\colon \mathbb Q
_{p}] = r(P)-2$ and $\mathcal{G}(E _{f}(p)/E _{f}) \cong
\mathcal{G}(E(p)/E)$, one obtains from (1.1) and (1.2) that
$\mathcal{G}(E(p)/E)$ is a Demushkin group and $r(p)_{E} = r(P)$.
Taking also into account that $\nu $ is the greatest integer for
which $E(\varepsilon )$ contains a primitive $p ^{\nu }$-th root of
unity, one deduces from \cite{D1}, Theorem~1 (see also \cite{Lab2},
Sect. 5), that the reduced component of $C(\mathcal{G}(E(p)/E))$ is
of order $p ^{\nu }$. Since $p > 2$ and by (1.3) (ii), $P$ is
uniquely determined by $(r(P), \nu )$, up-to an isomorphism, these
results indicate that $\mathcal{G}(E(p)/E) \cong P$, which proves
Proposition 6.6.
\end{proof}

\medskip

\section{\bf The decomposition groups of admissible Demushkin groups}

\par
\medskip
The purpose of this Section is to prove Theorem 1.3 and to describe
the decomposition groups of Demushkin pro-$p$-groups of finite rank.
The assumptions of Theorem 1.3 guarantee that the field $E ^{\prime
} = E(p) \cap E _{h(w)}$ is a proper extension of $E$, where $E
_{h(w)}$ is a Henselization of $E$ in $E _{\rm sep}$ relative to
$w$. Put $\Sigma (R) = R(\varepsilon ) \cap E ^{\prime }(\varepsilon
) ^{\ast p}$, for each $R \in {\rm I}(E(p)/E)$, and denote by $w
^{\prime }$ the $p$-Henselian valuation of $E ^{\prime }$ extending
$w$. Suppose first that $w(E)$ is Archimedean. We show that $E
^{\prime }/E$ is an infinite extension. Assuming the opposite, one
obtains the existence of a field $\Phi \in I(E(p)/E ^{\prime })$,
such that $\Phi /E$ is Galois, $[\Phi \colon E] \in \mathbb N$, and
$w$ has at least $2$ nonequivalent prolongations $w _{1}$ and $w
_{2}$ on $\Phi $. Since $w _{2} = w _{1} \circ \theta $, for some
$\theta \in \mathcal{G}(\Phi /E)$, and one can choose $w _{1} = w
^{\prime }_{\Phi }$, $w _{1}$ and $w _{2}$ must be $p$-Henselian
valuations. This contradicts the Grunwald-Wang theorem and thereby
proves that $[E ^{\prime }\colon E] = \infty $. Hence, by Remark
3.5, $\mathcal{G}(E(p)/E ^{\prime })$ is a free pro-$p$-group. Our
argument also relies on the fact that $E ^{\prime } _{h(w')}$ is
separably closed in $E ^{\prime }_{w'}$. In view of (3.6), this
ensures that $E ^{\prime \prime }_{\varepsilon } \subseteq N(L
^{\prime }/E ^{\prime \prime })$ whenever $L ^{\prime }$ is a finite
extension of $E ^{\prime }$ in $E(p)$ and $E ^{\prime \prime } \in
I(L ^{\prime }/E ^{\prime })$. Therefore, it can be easily deduced
from the Approximation Theorem (e.g., \cite{L}, Ch. XII, Sect. 1)
and norm transitivity in towers of finite separable extensions, that
$E _{\varepsilon } \subseteq N(L(\varepsilon )/E(\varepsilon
)).\Sigma (E)$, for every finite extension $L$ of $E$ in $E(p)$.
Consider the fields $E _{n} \in I(E(p)/E)$, $n \in \mathbb N$,
defined inductively as follows:
\par
\medskip
(7.1) $E _{1} = E$, and for each $n \in \mathbb N$, $E _{n+1}$ is
the compositum of the extensions of $E _{n}$ in $E(p)$ of degree
$p$.
\par
\medskip\noindent
It is clear from (7.1), Lemma 3.2 and Galois theory that $E _{n}/E$
is a finite Galois extension, and by Proposition 3.4,
$\mathcal{G}(E(p)/E _{n})$ is a Demushkin group of rank $2 + [E
_{n}\colon E](r(p)_{E} - 2)$, for each $n \in \mathbb N$. Note also
that Galois theory and the subnormality of proper subgroups of
finite $p$-groups imply that $E(p) = \cup _{n=1} ^{\infty } E _{n}$.
Observing that the symbol $E(\varepsilon )$-algebras $A
_{\varepsilon }(E(\varepsilon ); a _{1}, a _{2})$ and $A
_{\varepsilon }(E(\varepsilon ); a _{2}, a _{1})$ are inversely
isomorphic whenever $a _{1}, a _{2} \in E(\varepsilon ) ^{\ast }$,
one deduces from Corollary 3.3, statement (3.5) and Albert's theorem
that $N(E _{n+1} (\varepsilon )/E _{n}(\varepsilon )) \cap E
_{n,\varepsilon } \subseteq E _{n} (\varepsilon ) ^{\ast p}$. Denote
by $w _{n}$ the valuation of $E _{n}$ induced by $w ^{\prime }$, for
each index $n$. The preceding observations indicate that $\Sigma (E
_{n}) = E _{n,\varepsilon }$, so it follows from Albert's theorem and
the $p$-Henselity of $w ^{\prime }$ that $E _{n+1} \subset E ^{\prime
}$, $n \in \mathbb N$. This yields $E(p) \subseteq E ^{\prime }$ when
$w(E) \le \mathbb R$.
\par
\medskip
In order to deduce Theorem 1.3 in general, it suffices to prove the
following statements, assuming that $E ^{\prime } \neq E(p)$ and
$w(E)$ is non-Archimedean:
\par
\medskip
(7.2) (i) $w _{G(E)}$ is $p$-Henselian and finite extensions of $E$
in $E(p)$ are inertial relative to $w _{G(E)}$; in particular,
$\mathcal{G}(E _{G(E)}(p)/E _{G(E)}) \cong \mathcal{G}(E(p)/E)$;
\par
(ii) The valuation $\hat w _{G(E)}$ of $E _{G(E)}$ is Archimedean;
equivalently, $G(E) \subseteq H$, for every $H \in {\rm Is}_{w}(E)$,
$H \neq \{0\}$.
\par
\medskip\noindent
The rest of this Section is devoted to the proof of (7.2). Let $H
^{\prime }$ be an arbitrary group from Is$_{w'}(E ^{\prime })$. By
(2.1) (ii), the topologies on $E ^{\prime }$ induced by $w ^{\prime
}$ and by $w ^{\prime } _{H'}$ are equivalent, whence the
completions $E ^{\prime } _{w'}$ and $E ^{\prime } _{w' _{H'}}$ are
$E$-isomorphic. Note also that $H ^{\prime }$ is a maximal element
of Is$_{w'}(E ^{\prime })$ if and only if $H ^{\prime } \cap w(E)$
is maximal in Is$_{w}(E)$ (with respect to the orderings by
inclusion). As a first step towards the proof of (7.2), we prove the
following:
\par
\medskip
(7.3) With notation being as above, suppose that the group $H = H
^{\prime } \cap w(E)$ is maximal in Is$_{w}(E)$. Then:
\par
(i) $w _{H}$ is $p$-Henselian with char$(E _{H}) \neq p$;
\par
(ii) Finite extensions of $E$ in $E(p)$ are inertial relative to $w
_{H}$; in particular, $\mathcal{G}(E _{H}(p)/E _{H}) \cong
\mathcal{G}(E(p)/E)$;
\par
(iii) $E ^{\prime } _{H'}$ is a $p$-Henselization of $E _{H}$
relative to $\hat w _{H}$, and $\hat w ^{\prime }_{H'}$ is the
$p$-Henselian valuation of $E ^{\prime }_{H'}$ extending $\hat w
_{H}$.
\par
\medskip\noindent
Statement (2.1) (i) and the choice of $H$ ensure that $w ^{\prime
}_{H}(E ^{\prime })$ and $w _{H}(E)$ are Archimedean. Since, by
Proposition 2.1, $w ^{\prime } _{H'}$ is $p$-Henselian, these
observations enable one to deduce from the inequality $E ^{\prime }
\neq E(p)$ that $w _{H}$ is $p$-Henselian. As $w(E)$ is
non-Archimedean, we have $H \neq \{0\}$, which implies that $\hat w
_{H}$ is a nontrivial valuation of $E _{H}$. Therefore, $E _{H}$ is
infinite, so it follows from Proposition 5.1 and Lemma 4.3 that
char$(E _{H}) \neq p$. This completes the proof of (7.3) (i). The
latter conclusion of (7.3) (ii) follows from (2.3) (ii) and the
former one, and (7.3) (iii) is implied by (7.3) (i) and Proposition
2.1. It remains for us to prove former assertion of (7.3) (ii).
Denote for brevity by $(\Sigma , \sigma )$ some Henselization of
$(E, w _{H})$. Clearly, $(\Sigma , \sigma )/(E, w _{H})$ is
immediate. Moreover, by (2.1) (iii) and the maximality of $H$ in
Is$_{w}(E)$, $w _{H}(E)$ is Archimedean, which implies that $E$ is
dense in $\Sigma $ with respect to the topology of $\sigma $.
Observing also that $w _{H}$ is $p$-Henselian and arguing as in the
concluding part of the proof of (4.1), one obtains that $\Sigma (p)
= E(p)\Sigma $, $E(p) \cap \Sigma = E$ and $\mathcal{G}(E(p)/E)
\cong \mathcal{G}(\Sigma _{\sigma } (p)/\Sigma _{\sigma }) \cong
\mathcal{G}(\Sigma (p)/\Sigma )$. This implies the former statement
of (7.3) (ii) is equivalent to the one that finite extensions of
$\Sigma $ in $\Sigma (p)$ are inertial relative to $\sigma $. Thus
the proof of our assertion reduces to the special case in which $w
_{H}$ is Henselian. In this case, when $\varepsilon \notin E$, the
former part of (7.3) (ii) is contained in \cite{Ch2}, Lemma~1.1 (a).
Suppose finally that $w _{H}$ is Henselian and $\varepsilon \in E$.
By \cite{Ch3}, Lemma~3.8, and the assumption that
$\mathcal{G}(E(p)/E)$ is a Demushkin group, then $E$ is a
$p$-quasilocal field, in the sense of \cite{Ch3}. Since $r(p)_{E}
\ge 3$ and char$(E _{H}) \neq p$, this ensures that $w _{H}(E) = pw
_{H}(E)$ (apply \cite{Ch5}, (1.4)). The obtained result, combined
with Ostrowski's theorem, implies that finite extensions of $E$ in
$E(p)$ are inertial relative to $w _{H}$, so (7.3) is proved.

\medskip
\begin{rema}
In the setting of (7.2), let $w$ be of finite height $d \ge 2$,
and let $H$ and $H ^{\prime }$ be as in (7.3). Then $\hat w _{H}$
is of height $d - 1$, so (7.2) and Theorem 1.3 can be proved by
induction on $d$ (taking the inductive step via (7.3)).
\end{rema}

\medskip
The second step towards the proof of (7.2) and Theorem 1.3 is
contained in the following lemma.

\medskip
\begin{lemm}
Let $(E, w)$ be a field with $\mathcal{G}(E(p)/E)$ Demushkin,
$r(p)_{E} \in \mathbb N$, and $w _{H}$ non-$p$-Henselian, for any $H
\in {\rm Is}_{w}(E)$. Then $E(p) \subseteq E _{h(w)}$.
\end{lemm}

\begin{proof}
Put $E ^{\prime } = E(p) \cap E _{h(w)}$, denote by $w ^{\prime }$
the $p$-Henselian valuation of $E ^{\prime }$ extending $w$, and let
$E _{n}$, $n \in \mathbb N$, be the fields defined in (7.1). We
prove Lemma 7.2 by showing that $E _{n} \subseteq E ^{\prime }$, for
every $n \in \mathbb N$. It is clear from Galois theory and the
subnormality of proper subgroups of finite $p$-groups that $(E, w
_{G})$ will be $p$-Henselian, if $w _{G}$ is uniquely extendable on
each finite extension of $E$ in $E(p)$ of degree $p$. At the same
time, (7.3) and our assumptions indicate that $w(E)$ equals the
union $H(E)$ of the groups $H \in {\rm Is}_{w}(E)$. Since $E
_{h(v)}$ contains as an $E$-subalgebra a Henselization of $(E, w
_{H})$, for each $H \in {\rm Is}_{w}(E)$, this allows us, for the
proof of Lemma 7.2, to consider only the special case where
char$(\widehat E) \neq p$. Also, it follows from the equality $w(E)
= H(E)$ that if $R$ is a finite extension of $E$ in $E(p)$, and
$f(X) \in O _{w}(E) [X]$ is the minimal (monic) polynomial of some
primitive element of $R/E$, then the value of the discriminant of
$f(X)$ is contained in some $H _{f} \in {\rm Is}_{w}(E)$. When $H
\in {\rm Is}_{w}(E)$, $H _{f} \subseteq H$ and the reduction $f _{H}
[X] \in E _{H} [X]$ modulo $M _{w _{H}}(E)$ is irreducible over $E
_{H}$, this implies that $w _{H}$ is uniquely extendable to a
valuation of $R$. Therefore, the assumptions that $r(p)_{E} \in
\mathbb N$ and $w _{H}$ is not $p$-Henselian, for any $H \in {\rm
Is}_{w}(E)$ indicate that $R$ can be chosen so that $f _{H}$ is
reducible over $E _{H}$ whenever $H \in {\rm Is}_{w}(E)$ and $H _{f}
\le H$. It follows from the choice of $R$ that it is embeddable in
$E _{w}$ over $E$. Since the valuation $w ^{\prime }$ of $E ^{\prime
}$ is $p$-Henselian, whence $E ^{\prime }$ is separably closed in $E
^{\prime }_{w'}$, this implies that $R \subseteq E ^{\prime }$. The
choice of $R$ also ensures that $(R, w _{R})$ is immediate over $(E,
w)$, where $w _{R}$ is the valuation of $R$ induced by $w ^{\prime
}$. Let $\chi $ be a generator of $\mathcal{G}(R/E)$. It follows
from the immediacy of $(R, w _{R})/(E, w)$ and the equality
$[R\colon E] = p$ that the compositions $w _{R} \circ \chi ^{u-1}$,
$u = 0, \dots , p - 1$, are pairwise independent valuations of $R$
extending $w$ (see \cite{B1}, Ch. VI, Sect. 8.2, and \cite{L}, Ch.
IX, Proposition~11). As $\mathcal{G}(E(p)/R)$ is a Demushkin group
and $r(p)_{R} \in \mathbb N$, this enables one to deduce from
Grunwald-Wang's theorem that $(R, w _{R})$ satisfies the conditions
of Lemma 7.2. The obtained result allows us to define inductively a
tower of fields $R _{n} \in I(E ^{\prime }/E)$, $n \in \mathbb N$,
whose union $\Phi $ is an infinite extension of $E$ and embeds in $E
_{w}$ over $E$. By \cite{Lab1}, Theorem~2 (ii),
$\mathcal{G}(E(p)/\Phi )$ is a free pro-$p$-group, so it follows
from (3.6), Lemma 3.2 (i), the density of $E$ in $\Phi $ and the
Approximation Theorem that if $L \in I(E(p)/E)$, $[L\colon E] \in
\mathbb N$ and $\alpha \in E _{\varepsilon }$, then there is $\theta
\in L _{\varepsilon }$, such that $\bar w(-1 + N _{E(\varepsilon
)}^{L(\varepsilon )} (\theta )\alpha ^{-1}) > 0$, for each
prolongation $\bar w$ of $w$ on $E(\varepsilon )$. In view of the
$p$-Henselity of $w ^{\prime }$, the equality char$(\widehat E) = 0$
and Ostrowski's theorem, this ensures that $N _{E(\varepsilon
)}^{L(\varepsilon )} (\theta )\alpha ^{-1} \in E ^{\prime
}(\varepsilon ) ^{\ast p}$. These observations, combined with the
fact that $N(E _{n+1}(\varepsilon )/E _{n}(\varepsilon )) \cap E
_{n,\varepsilon } = E _{n}(\varepsilon ) ^{\ast p}$, for every $n
\in \mathbb N$, enable one to prove by induction on $n$ that $E
_{n,\varepsilon } \subseteq E ^{\prime }(\varepsilon ) ^{\ast p}$,
whence $E _{n} \subseteq E ^{\prime }$, for every index $n$. As
$E(p) = \cup _{n=1} ^{\infty } E _{n}$, this means that $E ^{\prime
} = E(p)$, as claimed.
\end{proof}

\medskip
We are now in a position to prove (7.2) and Theorem 1.3. Denote by
$\mathcal{H}(E)$ the set of those $H \in {\rm Is}_{w}(E)$, $H
\supseteq G(E)$, for which $w _{H}$ is $p$-Henselian, and for each
$H \in \mathcal{H}(E)$, let $H ^{\prime }$ be the preimage in $w
^{\prime }(E ^{\prime })$ of the maximal torsion subgroup of $w
^{\prime }(E ^{\prime })/H$. It is easily verified that $H ^{\prime
} \in {\rm Is}_{w'}(E ^{\prime })$ and $H ^{\prime } \cap w(E) = H$.
The assumption that $E ^{\prime } \neq E(p)$ and Lemma 7.2 indicate
that $\mathcal{H}(E)$ is nonempty. Note also that finite extensions
of $E$ in $E(p)$ are inertial relative to $w _{H}$ (and
$\mathcal{G}(E(p)/E) \cong \mathcal{G}(E _{H}(p)/E _{H})$), for each
$H \in \mathcal{G}(E)$. Since char$(E _{H}) \neq p$, this is
obtained from \cite{Ch2}, Lemma~1.1 (a), and \cite{Ch5}, (1.4), by
the method of proving (7.3) (ii). Taking further into account that
$\mathcal{H}(E)$ is closed under the formation of intersections, and
applying Zorn's lemma, one concludes that $\mathcal{H}(E)$ contains
a minimal element, say $G$, with respect to inclusion. We show that
$G = G(E)$. As $G ^{\prime } \in {\rm Is}_{w'}(E ^{\prime })$ and $G
^{\prime } \cap w(E) = G$, it follows from Proposition 2.1 and the
$p$-Henselian property of $w _{G}$ that $(E ^{\prime }_{G'}, \hat w
_{G'})$ is a $p$-Henselization of $(E _{G}, \hat w _{G})$.
Considering now $(E _{G}, \hat w _{G})$ instead of $(E, w)$, and
applying Proposition 2.1, one reduces the proof of (7.2) and Theorem
1.3 to the special case in which $w(E) = G$. Using Lemma 7.2, one
also concludes that if $G \neq G(E)$, then there exists a group
$\widetilde G \in \mathcal{H}(E)$ which is properly included in $G$.
This contradicts the minimality of $G$ in $\mathcal{H}(E)$ and so
proves that $G = G(E)$. In view of the preceding observations, the
obtained result completes the proof of (7.2) (i) (and of Theorem 1.3
in the case where char$(\widehat E) \neq p$). The equality $w(E) =
G(E)$ means that $w(p) \notin H$, for any $H \in {\rm Is}_{w}(E)$.
This ensures that char$(E _{H}) = p$, $H \in {\rm Is}_{w}(E)$, and
Is$_{w}(E)$ satisfies the conditions of Zorn's lemma with respect to
the ordering by inclusion. Taking a maximal element $\overline H$ in
Is$_{w}(E)$, one obtains from (7.3) (i) that $\overline H = 0$,
which completes the proof of (7.2). As $w$ is not $p$-Henselian, the
obtained result and Proposition 2.1 imply that $\hat w _{G(E)}$ is
not $p$-Henselian and $(E ^{\prime }_{G(E)'}, \hat w _{G(E)'})$ is a
$p$-Henselization of $(E _{G(E)}, \hat w _{G(E)})$. Since
$\mathcal{G}(E(p)/E) \cong \mathcal{G}(E _{G(E)}(p)/E _{G(E)})$ and
$\hat w_{G(E)}$ is Archimedean, this rules out the possibility that
$E ^{\prime } \neq E(p)$, so Theorem 1.3 is proved.

\medskip
\begin{rema}
Let $E$ be a field with $r(p)_{E} \in \mathbb N$, for some $p \in
P(E)$, and let $\mathbb F$ be the prime subfield of $E$. Suppose
that $\mathcal{G}(E(p)/E)$ is a Demushkin group and, in case
$r(p)_{E} \ge 3$, $E$ contains a primitive $p$-th root of unity. As
shown in \cite{Ch6}, then it follows from Proposition 3.4,
\cite{Wh}, Theorem~2, and (when $r(p) \ge 3$) from \cite{MS},
(16.1), and the isomorphism $H ^{2}(\mathcal{G}(E(p)/E), \mathbb F
_{p}) \cong \ _{p} {\rm Br}(E)$, that $E$ possesses a subfield
$\widetilde E$, such that $\mathcal{G}(\widetilde E(p)/\widetilde E)
\cong \mathcal{G}(E(p)/E)$, $\widetilde E$ is algebraically closed
in $E$ and $\widetilde E/\mathbb F$ is of finite transcendency
degree $d$. This ensures that nontrivial valuations of $\widetilde
E$ are of heights $\le d + 1$. Note also that, under the hypotheses
of Theorem 1.3, $\widetilde E$ can be chosen so that its valuation
$\tilde w$ induced by $w$ is nontrivial and non-$p$-Henselian. In
view of Remark 7.1, these observations simplify  the proof of
Theorem 1.3 and relate (1.5) to the problem of classifying
admissible Demushkin groups.
\end{rema}

\medskip
Theorem 1.3 allows one to view Theorem 1.1 (ii) as a generalization
of \cite{E2}, Theorem~6.3 (ii). Observe now that the proof of
Theorem 1.3 in case $w(E) \le \mathbb R$ and the one of Lemma 7.2 do
not require that $r(p)_{E} \ge 3$. Since, by Proposition 3.4, the
class of Demushkin groups of rank $2$ is closed under the formation
of open subgroups, this enables one to obtain the following result
by adapting the proof of Theorem 1.3:

\medskip
\begin{prop}
Let $(E, w)$ be a nontrivially valued field with $r(p)_{E} = 2$,
for some $p \in P(E)$, and suppose that $\mathcal{G}(E(p)/E)$ is a
Demushkin group and the field $E ^{\prime } = E _{h(w)} \cap E(p)$
is different from $E$ and $E(p)$. Then:
\par
{\rm (i)} $E _{h(w)}$ contains a primitive $p$-th root of unity,
$E ^{\prime }/E$ is a $\mathbb Z _{p}$-extension and $p \notin
P(\widehat E)$; in addition, if {\rm char}$(\widehat E) = p$, then
$\widehat E$ is perfect;
\par
{\rm (ii)} $\Delta (w) \neq p\Delta (w)$ and $w _{H}$ is
$p$-Henselian, for some $H \in {\rm Is}_{w}(E)$; in particular,
$w(E)$ is non-Archimedean.
\end{prop}

\medskip
It is not difficult to see that, for each Demushkin pro-$p$-group
$P$ of rank $2$, there exists $(E, w)$ satisfying the conditions
of Proposition 7.4 with
\par\noindent
$\mathcal{G}(E(p)/E) \cong P$. Note also that $\mathcal{G}(E(p)/E
^{\prime })$ is a characteristic subgroup of $\mathcal{G}(E(p)/E)$
unless $P \cong \mathbb Z _{p} ^{2}$. This follows from Galois
theory and the fact that $E ^{\prime }$ is the unique $\mathbb Z
_{p}$-extension of $E$ in $E(p)$. Observe finally that the
(continuous) automorphism group of $\mathbb Z _{p} ^{2}$ acts
transitively upon the set $\Omega _{p}$ of those closed subgroups
$\Gamma \le \mathbb Z _{p} ^{2}$, for which $\mathbb Z _{p}
^{2}/\Gamma \cong \mathbb Z _{p}$. This implies that if $P \cong
\mathbb Z _{p} ^{2}$ and $\Gamma \in \Omega _{p}$, then $\Gamma $
is a direct summand in $\mathbb Z _{p} ^{2}$, $\Gamma \cong
\mathbb Z _{p}$ and there is an isomorphism $P \cong
\mathcal{G}(E(p)/E)$ which maps $\Gamma $ upon $\mathcal{G}(E(p)/E
^{\prime })$. Thus Corollary 6.4 and Proposition 7.4, combined
with Theorems 1.1, 1.2 and 1.3, completely describe the
decomposition groups of admissible Demushkin pro-$p$-groups of
finite rank. The question of whether a field $R$, such that
$\mathcal{G}(R(p)/R)$ is a Demushkin group with $r(p)_{R} \in
\mathbb N$, possesses a $p$-Henselian valuation remains open. By
\cite{E2}, Proposition~9.4, the answer is affirmative, if
$r(p)_{R} = 2$ and $R$ contains a primitive $p$-th root of unity.


\medskip
\begin{thebibliography}{aa}


\bibitem{A1} A.A. Albert, \emph{Modern Higher Algebra}, Univ. of
Chicago Press, XIV, Chicago, Ill., 1937.

\bibitem{B1} N. Bourbaki, \emph{Elements de Mathematique. Algebre
Commutative}, Chaps. V, VI, VII, Hermann, Paris, 1964.

\bibitem{Ch} C. Chevalley, \emph{Sur la theorie du corps de classes
dans les corps finis et les corps locaux}, J. Fac. Sci. Univ. Tokio,
II, {\bf 9} (1933), 365-476.

\bibitem{Ch1} I.D. Chipchakov, \emph{Henselian valued  quasi-local
fields with totally indivisible value groups}, I, Comm. Algebra {\bf
27} (1999), 3093-3108; Preprint, arXiv:1011.2611v3 [math.Ra].

\bibitem{Ch2} I.D. Chipchakov, \emph{On the Galois cohomological
dimensions of stable fields with Henselian valuations}, Comm.
Algebra {\bf 30} (2002), 1549-1574.

\bibitem{Ch3} I.D. Chipchakov, \emph{On the residue fields of
Henselian valued stable fields}, I, J. Algebra {\bf 319} (2008),
16-49.

\bibitem{Ch4} I.D. Chipchakov, \emph{Primarily quasilocal fields and
$1$-dimensional abstract local class field theory}, Preprint,
arXiv:math/0506515v7 [math.RA].

\bibitem{Ch5} I.D. Chipchakov, \emph{On Henselian valuations and
Brauer groups of primarily quasilocal fields}, Preprint.

\bibitem{Ch6} I.D. Chipchakov, \emph{Characterization of Demushkin
groups among Galois groups of maximal $p$-extensions}, Preprint.

\bibitem{D1} S.P. Demushkin, \emph{The group of a maximal
$p$-extension of a local field}, Izv. Akad. Nauk SSSR Math. Ser.
{\bf 25} (1961), 329-346 (in Russian).

\bibitem{D2} S.P. Demushkin, \emph{On $2$-extensions of a local
field}, Sibirsk. Mat. Zh. {\bf 4} (1963), 951-955 (in Russian:
English transl. in Amer. Math. Soc. Transl., Ser. 2, {\bf 50}
(1966), 178-182).

\bibitem{DL} D. Dummit and J. Labute, \emph{On a new characterization
of Demushkin groups}, Invent. Math. {\bf 73} (1983), 413-418.

\bibitem{E1} I. Efrat, \emph{Finitely generated pro-$p$ Galois
groups of $p$-Henselian fields}, J. Pure Appl. Algebra {\bf 138}
(1999), 215-228.

\bibitem{E2} I. Efrat, \emph{Demuskin fields with valuations}, Math.
Z. {\bf 243} (2003), 333-353.

\bibitem{E3} I. Efrat, \emph{Valuations, Orderings, and Milnor
$K$-Theory,} Math. Surveys and Monographs 124, Amer. Math. Soc.,
Providence, RI, 2006.

\bibitem{EE} O. Endler and A.J. Engler, \emph{Fields with Henselian
valuation rings}, Math. Z. {\bf 152} (1977), 191-193.

\bibitem{EV} A.J. Engler and T.M. Viswanathan, \emph{Digging holes
in algebraic closures a la Artin}, I, Math. Ann. {\bf 265} (1983),
263-271; II, Contemp. Math. {\bf 8} (1982), 351-360.

\bibitem{Er} Yu.L. Ershov, \emph{Decision Problems and
Constructivizable Models}, Mat. Logika I Osnovaniya Matematiki,
Nauka, Moscow, 1980 (in Russian).

\bibitem{FJ} M.D. Fried and M. Jarden, \emph{Field Arithmetic}, 3rd
revised ed., Ergebnisse der Math. und ihrer Grenzgebiete, 3, Fol.
11, Berlin-Heidelberg-New York, Springer, 2008.

\bibitem{I} K. Iwasawa, \emph{Local Class Field Theory}, transl.
from Japanese into Russian by A.A. Bel'skii, transl. edited by B.B.
Venkov, Mir, Moscow, 1983.

\bibitem{JW} B. Jacob and A.R. Wadsworth, \emph{Division algebras
over Henselian fields}, J. Algebra {\bf 128} (1990), 126-179.

\bibitem{Ko1} H. Koch, \emph{Die Galoische Theorie der
$p$-Erweiterungen}, Math. Monogr. 10, VEB Deutscher Verlag der
Wissenschaften, Berlin; Springer-Verlag X, Berlin-Heidelberg-New
York, 1970.

\bibitem{Lab1} J.P. Labute, \emph{Demushkin groups of rank $\aleph
_{0}$}, Bull. Soc. Math. France {\bf 94} (1966), 211-244.

\bibitem{Lab2} J.P. Labute, \emph{Classification of Demushkin
groups}, Canad. J. Math. {\bf 19} (1967), 106-132.

\bibitem{La} T.Y. Lam, \emph{Orderings, valuations and quadratic forms},
Conf. Board Math. Sci. Regional Conf. Ser. Math. No 52, Amer. Math.
Soc., Providence, RI, 1983.

\bibitem{L} S. Lang, \emph{Algebra}, Reading, Mass., Addison-Wesley,
Inc., XVIII, 1965.

\bibitem{LR} F. Lorenz and P. Roquette, \emph{The theorem of
Grunwald-Wang in the setting of valuation theory}, F.-V. Kuhlmann
(ed.) et. al., Valuation theory and its applications, Vol. II
(Saskatoon, SK, 1999), 175-212, Fields Inst. Commun., {\bf 33},
Amer. Math. Soc., Providence, RI, 2003.

\bibitem{MT} O.V. Mel'nikov and O.I. Tavgen', \emph{The absolute
Galois group of a Henselian field}, Dokl. Akad. Nauk BSSR {\bf 29}
(1985), 581-583 (in Russian).

\bibitem{MS} A.S. Merkur'ev and A.A. Suslin, \emph{$K$-cohomology of
Brauer-Severi varieties and the norm residue homomorphism}, Izv.
Akad. Nauk SSSR {\bf 46} (1982), 1011-1046  (Russian: English
transl. In Math. USSR Izv. {\bf 21} (1983), 307-340).

\bibitem{P} R. Pierce, \emph{Associative Algebras}, Graduate Texts
in Math., vol. 88, Springer-Verlag, New York-Heidelberg-Berlin,
1982.

\bibitem{Pop} F. Pop, \emph{Galoissche Kennzeichnung p-adisch
abgeschlossener K\"orper}, J. Reine Angew. Math. {\bf 392} (1988),
145-175.

\bibitem{S1} J.-P. Serre, \emph{Structure de certains
pro-$p$-groupes (apr\`{e}s Demushkin)}, Semin. Bourbaki {\bf 15}
1962/1963 (1964), No 252.

\bibitem{S2} J.-P. Serre, \emph{Cohomologie Galoisienne}, Cours au
College de France, 1962/1963, 2nd ed., Springer-Verlag,
Berlin-G\"ottingen-Heidelberg-New York, 1964.

\bibitem{S3} J.-P. Serre, \emph{Sur la dimension cohomologique des
groupes profinis}, Topology {\bf 3} (1965), 413-420.

\bibitem{Sh} I.R. Shafarevich, \emph{On $p$-extensions}, Mat. Sb.,
N. Ser., {bf 20} ({\bf 62}) (1947), 351-363 (in Russian).

\bibitem{Ta} J. Tate, \emph{Relations between $K _{2}$ and Galois
cohomology}, Invent. Math. {\bf 36} (1976), 257-274.

\bibitem{TY} I.L. Tomchin and V.I. Yanchevskij, \emph{On defects of
valued division algebras}, Algebra i Analiz {\bf 3} (1991), 147-164
(in Russian: English transl. in St. Petersburg Math. J. {\bf 3}
(1992), 631-647).

\bibitem{W2} R. Ware, \emph{Galois groups of maximal
$p$-extensions}, Trans. Amer. Math. Soc. {\bf 333} (1992), 721-728.

\bibitem{Wh} G. Whaples, \emph{Algebraic extensions of arbitrary
fields}, Duke Math. J. {\bf 24} (1957), 201-204.
\end{thebibliography}
\end{document}